\documentclass[reqno]{amsart}
\usepackage{amssymb}

\newtheorem{theorem}{Theorem}[section]
\newtheorem{proposition}[theorem]{Proposition}
\newtheorem{lemma}[theorem]{Lemma}

\newtheorem{definition}[theorem]{Definition}

\theoremstyle{remark}
\newtheorem{remark}[theorem]{Remark}

\numberwithin{equation}{section}

\begin{document}

\title[Macdonald Polynomials with Unitary Parameters]
{Orthogonality of Macdonald Polynomials\\ with Unitary Parameters}

\author{J.F.  van Diejen}

\author{E. Emsiz}

\address{
Facultad de Matem\'aticas, Pontificia Universidad Cat\'olica de Chile,
Casilla 306, Correo 22, Santiago, Chile}
\email{diejen@mat.puc.cl, eemsiz@mat.puc.cl}


\thanks{Work was supported in part by the {\em Fondo Nacional de Desarrollo
Cient\'{\i}fico y Tecnol\'ogico (FONDECYT)} Grants \# 1130226 and  \# 11100315,
and by the {\em Anillo ACT56 `Reticulados y Simetr\'{\i}as'}
financed by the  {\em Comisi\'on Nacional de Investigaci\'on
Cient\'{\i}fica y Tecnol\'ogica (CONICYT)}}

\date{December 2012}

\begin{abstract}
For any admissible pair of irreducible reduced crystallographic root systems, we present discrete orthogonality relations for a finite-dimensional system of Macdonald polynomials with parameters on the unit circle subject to a truncation relation.
\end{abstract}

\maketitle

\section{Introduction}
In \cite{die-vin:quantum} a finite-dimensional system of discrete orthogonality relations was found for the
Macdonald polynomials \cite{mac:symmetric} with parameters of the form
$t=q^\text{g}$ and $q=e^{\frac{2\pi i}{(n+1)\text{g}+c}}$, where $n+1$ denotes the number of variables, $\text{g}$ is a positive real parameter, and $c$ is a positive integer (so both $t$ and $q$ lie on the unit circle and satisfy the truncation relation $t^{n+1}q^c=1$).
For $\text{g}$ integral Macdonald's parameters $q$ and $t$ become roots of unity; the discrete orthogonality relations of \cite{die-vin:quantum} specialize in this situation to those considered in \cite[\text{Sec}.~5]{kir:inner}. In particular, when $\text{g}=1$ elementary orthogonality relations for systems of periodic Schur polynomials are recovered, cf. \cite[\S 13.8]{kac:infinite}, \cite[\text{Sec.}~6]{kir:inner}, \cite[\text{Sec}.~4.2]{die:finite-dimensional}, and
\cite[\text{Sec}.~6.2]{kor-stro:slnk}.

The purpose of the present paper is to generalize the finite-dimensional discrete orthogonality relations of \cite{die-vin:quantum} to Macdonald polynomials with unitary parameters associated with arbitrary admissible pairs of irreducible reduced crystallographic root systems
\cite{mac:orthogonal} in the spirit of
\cite{die-sto:multivariable}, where the case of nonreduced root systems was considered.
We thus arrive at a parameter deformation interpolating between discrete orthogonality relations for such Macdonald polynomials with parameter values at roots of unity
\cite[Sec.~5]{che:macdonalds}, containing  as a special case elementary orthogonality relations for systems of periodic Weyl characters, cf.
\cite[\S 13.8]{kac:infinite}, \cite[\text{Sec.}~6]{kir:inner}, \cite[\text{Sec}.~5.3]{hri-pat:discretization} and \cite[\text{Sec}.~8.4]{die-ems:discrete}.

When the rank of the reduced root system is {\em one}, the orthogonality  considered here reduces to a finite-dimensional discrete orthogonality relation for
the $q$-ultraspherical polynomials \cite[\text{Sec}.~3C2]{rui:integrable}, \cite[\text{Sec}.~5.2]{die-vin:quantum} that arises as a parameter specialization of the celebrated orthogonality for the $q$-Racah polynomials found by Askey and Wilson \cite{koe-les-saw:hypergeometric}. The full $q$-Racah orthogonality corresponds from this perspective to the nonreduced rank-one situation \cite{die-sto:multivariable}. A very different multivariate analog of the  $q$-Racah polynomials was studied recently by Gasper, Rahman and Iliev \cite{gas-rah:some,ili:bispectral}.

The paper is organized as follows.
Section \ref{sec2} sets up notation and recalls
the definition of the Macdonald polynomials diagonalizing
Macdonald's difference operators.
In Section \ref{sec3} we formulate our finite-dimensional system of discrete orthogonality relations for the Macdonald polynomials with unitary parameters (subject to a truncation relation). The remainder of the paper is devoted to the proof of these orthogonality relations.
Specifically, we infer the orthogonality of the Macdonald polynomials with unitary parameters exploiting explicit formulas for the Macdonald difference operators
from \cite{die-ems:generalized} allowing to confirm the self-adjointness and the nondegeneracy of the spectrum in the present setting (Section \ref{sec4}).
Next, we rely on the duality symmetry
\cite{mac:affine,che:double,hai:cherednik,cha:algebraic}
to compute the norms of the Macdonald polynomials via Macdonald's Pieri-type recurrence relations associated with the (quasi-)minuscule weights \cite{mac:affine,las:inverse} (Section \ref{sec5}).
The total mass of the weight function is expressed compactly in product form by applying a finitely truncated version \cite{die:certain,maz:finite} of a basic hypergeometric summation formula associated with root systems due to Aomoto, Ito and Macdonald \cite{aom:elliptic,ito:symmetry,mac:formal}.
Some technical issues concerning the proof of the nondegeneracy of the eigenvalues of the Macdonald difference operators for exceptional root systems  are relegated to an appendix at the end of the paper.

\section{Preliminaries}\label{sec2}
In this section the definitions of the Macdonald polynomials and the Macdonald difference operators are recalled briefly.
For more detailed discussions with proofs and further background material the reader is referred to the standard texts \cite{mac:orthogonal,mac:affine,kir:lectures,che:double,hai:cherednik}.
Throughout familiarity with the basic properties of root systems and their Weyl groups \cite{bou:groupes} will be assumed.

\subsection{Macdonald polynomials}
Let $E$ be a real finite-dimensional Euclidean vector space with inner product $\langle \cdot ,\cdot \rangle$ spanned by an irreducible reduced crystallographic root system $R$.
We write $Q$, $P$, and $W$, for the root lattice, the weight lattice, and the Weyl group associated with $R$. The semigroup of the root lattice generated by a (fixed) choice of positive roots $R^+$
is denoted by $Q^+$ whereas  $P^+$  stands for the corresponding cone of dominant weights.
The group algebra
$\mathbb{C}[P]$ is spanned by formal exponentials
$e^\lambda$, $\lambda\in P$ characterized by the relations $e^0=1$,
$e^\lambda e^\mu=e^{\lambda +\mu}$. The Weyl group acts linearly on $\mathbb{C}[P]$
via $we^\lambda :=e^{w\lambda}$, $w\in W$, and the $W$-invariant subalgebra $\mathbb{C}[P]^W$ is spanned by the basis of symmetric monomials $m_\lambda := \sum_{\mu \in W\lambda} e^\mu$, $\lambda\in P^+$ (where the sum is meant over the $W$-orbit of $\lambda$). This monomial basis inherits a partial order from the dominance order on
the cone of dominant weights:
\begin{equation}
\mu\leq \lambda\quad \text{iff}\quad \lambda-\mu\in Q^+\qquad (\lambda ,\mu\in P^+).
\end{equation}

The dual root system $R^\vee:=\{\alpha^\vee\mid \alpha\in R\}$ and its
positive subsystem $R^{\vee ,+}$ are obtained from $R$ and $R^+$ by applying the involution
\begin{equation}\label{check-inv}
x\mapsto x^\vee:=2x/\|x\|^2\qquad (x\in E\setminus\{0\} ) ,
\end{equation}
where $\|x\|:=\langle x ,x \rangle^{1/2}$.
Following Macdonald's terminology, we refer to a tuple $(R,\hat{R})$
with $\hat{R}$ being equal either to $R$ or $R^\vee$ as an {\em admissible pair} of root systems \cite{mac:orthogonal}.
The assignment
$a\mapsto\hat{a}$ for $a\in R\cup\hat{R}$, where $\hat{a}:=a$ if $\hat{R}=R$ and $\hat{a}:=a^\vee$ if $\hat{R}=R^\vee$, defines an involution on $R\cup\hat{R}$ swapping the roots of $R$ and $\hat{R}$.
We write $\hat{Q}$ and $\hat{P}$ for the root lattice and the weight lattice of $\hat{R}$, and
more generally: dual objects associated with the admissible pair $(\hat{R},R)$ (i.e. with the role of $R$ and $\hat{R}$ interchanged) will be endowed with a superscript hat.

For a formal series $f=\sum_{\lambda\in P} f_\lambda e^\lambda$, $f_\lambda\in \mathbb{C}$, we define $\int f := f_0$ and $\bar{f}:=\sum_{\lambda\in P}\bar{f}_\lambda e^{-\lambda}$ (with $\bar{f}_\lambda$ meaning the complex conjugate of $f_\lambda$).
The Macdonald inner product on $\mathbb{C}[P]$  is then given by \cite{mac:orthogonal,mac:affine}
\begin{subequations}
\begin{equation}\label{ipa}
\langle f,g\rangle_\delta := |W|^{-1} \int  f \bar{g}\, \delta \qquad (f,g\in\mathbb{C}[P]) ,
\end{equation}
where $|W|$ stands for the order of $W$ and
\begin{equation}\label{ipb}
\delta:=\delta_{{}^+}\bar{\delta}_{{}^+}, \qquad \delta_{{}^+} := \prod_{\alpha \in R^+} \frac{( e^\alpha;q_\alpha)_\infty}{( t_\alpha e^\alpha;q_\alpha)_\infty}
\end{equation}
(employing standard notation for the
$q$-shifted factorial $(a;q)_m :=\prod_{k=0}^{m-1} (1-aq^k)$ with $m$ nonnegative integral or $\infty$). Here $q$ is a parameter taking values in
$(0,1)$ and
\begin{equation}
q_a:=q^{u_a},\qquad
t_a:=q_a^{\text{g}_a} \qquad (a\in R\cup\hat{R}),
\end{equation}
\end{subequations}
where $u_a:=\|\hat{a}\|/\|a^\vee \|$ (so $u_{\hat{a}}=u_a$) and
$\text{g}:R\cup\hat{R}\to (0,\infty)$ denotes a root multiplicative function such that $\text{g}_{wa}=\text{g}_a$ and $\text{g}_{\hat{a}}=\text{g}_a$ (for all
$w\in W$ and $a\in R\cup\hat{R}$). We think of this function as a
(pair of) positive parameter(s) attached to the
$(W\times\mathbb{Z}_2)$-orbits of $R\cup\hat{R}$ (where the  $\mathbb{Z}_2$-action corresponds to the `hat'-involution).

\begin{definition}[Macdonald Polynomials \cite{mac:orthogonal,mac:affine}]
For $\lambda\in P^+$, the {\em Macdonald polynomial} is defined as the unique element in $\mathbb{C}[P]^W$ of the form
\begin{subequations}
\begin{equation}\label{mp-d1}
p_\lambda = m_\lambda+\sum_{\mu <  \lambda } c_{\lambda \mu} (q,t)\, m_{\mu}
\end{equation}
with $c_{\lambda \mu} (q,t)\in\mathbb{C}$ such that
  \begin{equation}\label{mp-d2}
  \langle p_\lambda , m_{\mu } \rangle_\delta =0 \quad \text{for all} \ \mu < \lambda .
  \end{equation}
 \end{subequations}
\end{definition}

The elements of the group algebra will be interpreted as functions on $E$ through the evaluation homomorphism
$e^\lambda (x):= q^{\langle \lambda ,x\rangle}$,  $x\in E$.
Three remarkable fundamental properties of Macdonald's polynomials are the {\em orthogonality relations}
\cite[\text{Eqs.}~(5.3.4),(5.8.17)]{mac:affine}
\begin{subequations}
\begin{equation}\label{macdonald-orthogonality}
\langle p_\lambda , p_\mu\rangle_\delta  =
\begin{cases}
0&\text{if}\ \lambda \neq \mu \\
\frac{\hat{\delta}_{{}^+}(\rho_{\text{g}}+\lambda)}{\hat{\delta}_{{}^-}(\rho_{\text{g}}+\lambda ) }&\text{if}\ \lambda =\mu
\end{cases}
 ,
\end{equation}
the {\em principal specialization formula}  \cite[\text{Eq.}~(5.3.12)]{mac:affine})
\begin{equation}\label{specialization-formula}
\begin{split}
p_\lambda (\hat{\rho}_{\text{g}})&= \frac{\hat{\delta}_{{}^+}(\rho_{\text{g}}+\lambda)}{\hat{\delta}_{{}^+}(\rho_{\text{g}})e^\lambda (\hat{\rho}_{\text{g}})}
\\
 &=
\prod_{\alpha\in R^+}
t_\alpha^{-\langle\lambda,\alpha^\vee\rangle/2}
\frac{(t_\alpha q_\alpha^{\langle  \rho_{\text{g}},\alpha^\vee \rangle };   q_\alpha)_{\langle\lambda, \alpha^\vee \rangle}}{( q_\alpha^{\langle  \rho_{\text{g}},\alpha^\vee \rangle};q_\alpha)_{\langle \lambda, \alpha^\vee\rangle}} ,
 \end{split}
\end{equation}
and the
{\em duality symmetry} \cite[\text{Eq.}~(5.3.6)]{mac:affine} for the normalized Macdonald polynomials
$P_\lambda := p_\lambda/p_\lambda (\hat{\rho}_{\text{g}}) $:
\begin{equation}\label{duality-symmetry}
P_\lambda (\hat{\rho}_{\text{g}}+\mu)= \hat{P}_\mu (\rho_{\text{g}}+\lambda) \qquad
  (\lambda\in P^+,\mu\in \hat{P}^+),
\end{equation}
\end{subequations}
where we have employed the additional notation
\begin{equation*}
 \rho_{\text{g}}:=\frac{1}{2}\sum_{\alpha\in R^+} \text{g}_\alpha \alpha,\qquad
\delta_{{}^-} := \prod_{\alpha \in R^+} \frac{( t_\alpha^{-1}q_\alpha e^\alpha;q_\alpha)_\infty}{( q_\alpha e^\alpha;q_\alpha)_\infty} .
\end{equation*}

\subsection{Macdonald difference operators}
Macdonald's polynomials are joint eigenfunctions of an algebra of commuting difference operators that is isomorphic to $\mathbb{C}[\hat{P}]^W$, cf. \cite[\text{Eqs.}~(4.4.12), (5.3.3)]{mac:affine}. Explicit formulas for the difference operators associated with the basis elements $\hat{m}_\omega$ with $\omega\in\hat{P}^+\setminus \{0\}$ {\em minuscule} (i.e. $\langle \omega,\alpha^\vee\rangle\leq 1$ for all $\alpha\in \hat{R}^+$) or {\em quasi-minuscule} (i.e. $\langle \omega,\alpha^\vee\rangle\leq 1$ for all $\alpha\in \hat{R}^+\setminus\{\omega\}$ and $\omega$ is not minuscule)  were presented in
\cite[\text{Secs.}~5, 6]{mac:orthogonal}. In \cite[\text{Sec.}~3]{die-ems:generalized}
more general explicit formulas for the Macdonald difference operators  can be found corresponding to the subspace of
$\mathbb{C}[\hat{P}]^W$ spanned by the monomials $\hat{m}_\omega$ with $\omega\in \hat{P}^+$ {\em small} (viz.  $\langle \omega,\alpha^\vee\rangle\leq 2$ for all $\alpha\in \hat{R}^+$). To formulate the latter difference operators some further notation is needed.
Given $\lambda\in P^+$, let us denote the saturated set in $P$ cut out by the convex hull of $W\lambda$ by
\begin{equation}
P(\lambda):=\bigcup_{\mu\in P^+,\, \mu\leq \lambda} W\mu .
\end{equation}
The stabilizer of $x\in E$ in $W$ and the corresponding parabolic subsystem of $R$
are given by $W_x:=\{ w\in W\mid wx=x\}$ and $R_x:=\{\alpha\in R\mid \langle x,\alpha\rangle=0\}$ (with $R_x^+:=R_x\cap R^+$). Finally, for $\lambda\in P$ we write $w_\lambda$ for the unique shortest Weyl group element mapping $\lambda$ into the dominant cone $P^+$.

For $\omega\in \hat{P}^+$ small in the sense above, the Macdonald difference operator $\mathcal{D}_\omega$ from \cite{die-ems:generalized}
acts on meromorphic functions $f:E\to\mathbb{C}$ as
\begin{subequations}
\begin{equation}\label{D}
(\mathcal{D}_\omega f)(x)=\sum_{\substack{\nu\in  \hat{P}(\omega)\\ \eta\in W_\nu(w_\nu^{-1}\omega)   }}
V_\nu(x)U_{\nu,\eta}(x)(T_\nu f)(x) ,
\end{equation}
where $(T_v f)(x):=f(x+\nu)$ and the coefficients $V_\nu$ and $U_{\nu,\eta}$ are of the form
\begin{equation}\label{V}
  V_{\nu}(x)
 =\prod_{\substack{\alpha\in \hat{R}\\     \langle \nu,\alpha^\vee\rangle> 0 }}
\frac{\sin\kappa_\alpha(\langle x,\alpha^\vee \rangle+\text{g}_\alpha)}{\sin\kappa_\alpha\langle x,\alpha^\vee \rangle}
\prod_{\substack{\alpha\in \hat{R}\\     \langle \nu,\alpha^\vee\rangle=2 }}
\frac{\sin\kappa_\alpha(\langle x,\alpha^\vee\rangle+1+\text{g}_\alpha)}{\sin\kappa_\alpha(\langle x,\alpha^\vee \rangle+1)}
  \end{equation}
and
  \begin{equation}\label{U}
U_{\nu,\eta}(x)
=\prod_{\substack{\alpha\in \hat{R}_\nu\\     \langle \eta,\alpha^\vee\rangle> 0 }}
\frac{\sin\kappa_\alpha(\langle x,\alpha^\vee \rangle+\text{g}_\alpha)}{\sin\kappa_\alpha\langle x,\alpha^\vee \rangle}
\prod_{\substack{\alpha\in \hat{R}_\nu\\     \langle \eta,\alpha^\vee\rangle=2 }}
\frac{\sin\kappa_\alpha(\langle x,\alpha^\vee \rangle+1-\text{g}_\alpha)}{\sin\kappa_\alpha(\langle x,\alpha^\vee \rangle+1)}.
%    \end{split}
  \end{equation}
\end{subequations}
Here $\kappa$ is a (for now positive imaginary) parameter that is related to Macdonald's parameter $q$ via
\begin{subequations}
\begin{equation}\label{q}
q=\exp (2i\kappa)
\end{equation}
and $\kappa_a:=\kappa u_a$, so
\begin{equation}\label{qta}
q_a=\exp (2i\kappa_a)\quad\text{and}\quad t_a=\exp (2i\kappa_a \text{g}_a)\qquad (a\in R\cup\hat{R}).
\end{equation}
\end{subequations}

One has that \cite[\text{Thm.}~3.1]{die-ems:generalized}
\begin{subequations}
\begin{equation}\label{EV:omega}
\mathcal{D}_\omega p_\lambda=E_\omega(\rho_{\text{g}}+\lambda )p_\lambda ,
  \end{equation}
where
\begin{equation}
\label{spectrum}
\begin{split}
E_\omega &:=\hat{m}_\omega+\sum_{\mu\in\hat{P}^+,\, \mu<\omega}  \epsilon_{\omega ,\mu} \hat{m}_\mu\in\mathbb{C}[\hat{P}] ,\\
\epsilon_{\omega ,\mu} &:= \sum_{\eta\in W_\mu\omega}
\prod_{\substack{\alpha\in \hat{R}_\mu^+ \\  |\langle\eta ,\alpha^\vee\rangle |=1}} t_\alpha^{\langle\eta,\alpha^\vee\rangle} .
\end{split}
\end{equation}
\end{subequations}

When $\omega$ is minuscule $\hat{P}(\omega)=W\omega$ and
\begin{subequations}
\begin{equation}\label{D:m}
\mathcal{D}_\omega=\sum_{\nu\in  W\omega} V_\nu(x)T_\nu, \qquad E_\omega=\hat{m}_\omega ,
\end{equation}
whereas when $\omega$ is quasi-minuscule $\hat{P}(\omega)=W\omega\cup \{0\}$ and
\begin{equation}\label{D:qm}
\mathcal{D}_\omega=\sum_{\nu\in  W\omega} (U_{0,\nu}(x)+V_\nu(x)T_\nu ) ,\qquad E_\omega =\hat{m}_\omega +\epsilon_{\omega ,0} .
\end{equation}
\end{subequations}
The difference equation \eqref{EV:omega} reduces in these two situations to the explicit eigenvalue equations for the Macdonald polynomials stemming from \cite[\text{Sec}.~5]{mac:orthogonal} and \cite[\text{Sec}.~6]{mac:orthogonal}, respectively.

\subsection{Generic complex parameters}
From a Taylor expansion of $\hat{m}_\omega (\rho_{\text{g}}+\lambda)$ in $\kappa$:
\begin{align*}
\hat{m}_\omega (\rho_{\text{g}}+\lambda) &= \sum_{\nu\in W\omega} \exp({2i\kappa \langle \nu,\rho_{\text{g}}+\lambda \rangle})\\
&= |W\omega|-2\kappa^2\sum_{\nu\in W\omega} \langle \nu,\rho_{\text{g}}+\lambda \rangle^2 + O(\kappa^3) \\
&= |W\omega| \left( 1-\frac{2\kappa^2}{n}\|\omega\|^2 \|\rho_{\text{g}}+\lambda \|^2 \right) + O(\kappa^3),
\end{align*}
one reads-off that $\hat{m}_\omega (\rho_{\text{g}}+\mu)\neq \hat{m}_\omega (\rho_{\text{g}}+\lambda)$ as an analytic expression in the parameters  $\kappa$ and $\text{g}$ when
$\mu<\lambda$ (because the latter inequality implies that $\|\rho_{\text{g}}+\mu\|^2<\|\rho_{\text{g}}+\lambda \|^2$ for $\text{g}>0$). Here $n$ refers to the rank of the root system ($:=\dim E$) and $|W\omega|$ stands for the order of the orbit $W\omega$.
The upshot is that---for generic values in our parameter domain---the Macdonald polynomial $p_\lambda$ can be alternatively characterized as the unique polynomial of the form in Eq. \eqref{mp-d1} satisfying the eigenvalue equation \eqref{EV:omega} for the Macdonald operator $\mathcal{D}_\omega$ with $\omega$ (quasi-)minuscule, i.e. one then has explicitly
(cf. \cite[\text{Sec.}~4]{mac:orthogonal}):
\begin{equation}\label{P:explicit}
p_\lambda=
\biggl(
\prod_{\substack{\mu\in P^+\\  \mu < \lambda}} \frac{\mathcal{D}_\omega-E_\omega (\rho_{\text{g}}+\mu)}{E_\omega (\rho_{\text{g}}+\lambda)-E_\omega (\rho_{\text{g}}+\mu)}
\biggr)
m_\lambda
\end{equation}
(where  $\omega$ is assumed to be (quasi-)minuscule).

It is immediate from \eqref{P:explicit} and the explicit formulas for $\mathcal{D}_\omega$ and $E_\omega$ in \eqref{D:m} and \eqref{D:qm} that the Macdonald polynomial $p_\lambda$ is meromorphic in the parameters $\kappa$ and $\text{g}$. Hence, the Macdonald polynomial $p_\lambda$ extends meromorphically in these parameters to an eigenpolynomial of the form in \eqref{mp-d1} solving the eigenvalue equation \eqref{EV:omega}, \eqref{spectrum} for generic complex parameter values (and $\omega\in \hat{P}^+$ small).

\section{Main result: orthogonality relations for unitary parameters}\label{sec3}
 We exploit the  above meromorphy of the Macdonald polynomials in the parameters
 to continue the parameter $\kappa$ analytically from the imaginary to the real axis while leaving $\text{g}$ positive (so $q$, $q_a$ and $t_a$ become unitary) and
consider the normalized Macdonald polynomials satisfying the duality symmetry
\eqref{duality-symmetry}:
\begin{subequations}
\begin{equation}\label{P-reno}
P_\lambda = c_\lambda p_\lambda
\end{equation}
with
\begin{align}\label{specialization:trigonometric}
c_\lambda:=1/p_\lambda(\hat{\rho}_{\text{g}})\stackrel{\text{Eq.}~\eqref{specialization-formula}}{=}\prod_{\alpha\in R^+}
 \frac{(\langle \rho_{\text{g}},\alpha^\vee \rangle \!  :\! \kappa_\alpha)_{\langle\lambda,\alpha^\vee\rangle}}
         {(\langle \rho_{\text{g}},\alpha^\vee \rangle+\text{g}_\alpha \! :\! \kappa_\alpha)_{\langle \lambda,\alpha^\vee\rangle}} .
\end{align}
\end{subequations}
Here we  have employed trigonometric Pochhammer symbols of the form
\begin{equation}
(a\! :\!\kappa)_l:=2^l\sin (\kappa a) \sin \kappa(a+1)\cdots \sin\kappa(a+l-1)
\end{equation}
when $l$ is positive integral, subject to the convention that $(a\! :\!\kappa)_0:= 1$.
Following the standard conventions for Pochhammer symbols, we will occasionally abbreviate products of the form
$(a_1\! :\!\kappa)_l\cdots (a_k\! :\!\kappa)_l$ by
$(a_1,\ldots,a_k\! :\!\kappa )_l$.

For $c$ nonnegative integral, let
\begin{equation}\label{Pc}
P_c
 :=\{\lambda\in P^+ \mid \langle \lambda,\hat\psi^\vee \rangle \le  c\} \quad\text{and so}\quad
\hat P_c
         = \{\mu \in \hat P^+ \mid \langle \mu,\hat \varphi^\vee \rangle \le  c\} ,
\end{equation}
where $\varphi$ and $\psi$ refer to the maximal roots of $R$ and $\hat{R}$, respectively. Let us furthermore denote the maximal {\em short} root of $R$  by $\vartheta$ (with the convention that all roots of $R$ are short if the root system is simply laced). (So $\vartheta$ is equal to the unique quasi-minuscule weight of $R$ and $\vartheta^\vee$ is the maximal coroot in $R^\vee$.) From now it will be always assumed---unless explicitly stated otherwise---that $c>1$ and that {\em for $R$ of type $E_7$ the value of $c$ is not a proper multiple of $6$} (cf. Remark \ref{E7-6c:rem} below).

\begin{theorem}[Finite-Dimensional Discrete Orthogonality Relations]\label{main:thm}
For
 \begin{subequations}
\begin{equation}\label{special-parameter}
\kappa
 =\frac{\pi}{u_{\varphi} (h_{\text{g}}+c)}\quad\text{with}\quad h_{\text{g}}  := \langle \rho_{\text{g}}, \hat \psi^\vee\rangle +  \text{g}_\psi ,
\end{equation}
the Macdonald polynomials $P_\lambda$, $\lambda\in P_c$ are analytic in $\text{g}>0$
and satisfy the orthogonality relations
\begin{equation}\label{ort-rel1}
\sum_{\lambda\in P_c} P_\lambda (\hat{\rho}_{\text{g}}+\mu )\overline{P_\lambda (\hat{\rho}_{\text{g}}+\tilde\mu )}\Delta (\lambda)=
\begin{cases}
0&\text{if}\ \mu \neq \tilde\mu \\
\frac{\mathcal{N}_0}{\hat{\Delta}(\mu)}&\text{if}\ \mu =\tilde\mu
\end{cases}
\end{equation}
($\mu,\tilde{\mu}\in \hat{P}_c$), or equivalently
\begin{equation}\label{ort-rel2}
\sum_{\mu\in \hat{P}_c} P_\lambda (\hat{\rho}_{\text{g}}+\mu )\overline{P_{\tilde\lambda} (\hat{\rho}_{\text{g}}+\mu )}\hat{\Delta} (\mu )=
\begin{cases}
0&\text{if}\ \lambda \neq \tilde\lambda \\
\frac{\mathcal{N}_0}{\Delta (\lambda )} &\text{if}\ \lambda =\tilde\lambda
\end{cases}
\end{equation}
($\lambda,\tilde{\lambda}\in P_c$), where
\begin{align}\label{ort-measure}
\Delta(\lambda) :=\prod_{\alpha\in R^+}
 \frac{\sin\kappa_\alpha \langle \rho_{\text{g}}+\lambda,\alpha^\vee\rangle}{\sin\kappa_\alpha \langle \rho_{\text{g}},\alpha^\vee\rangle}
 \frac{(\langle \rho_{\text{g}},\alpha^\vee \rangle+\text{g}_\alpha\! :\!\kappa_\alpha)_{\langle\lambda,\alpha^\vee\rangle}}
         {(\langle \rho_{\text{g}},\alpha^\vee \rangle+1-\text{g}_\alpha\! :\!\kappa_\alpha)_{\langle \lambda,\alpha^\vee\rangle}} ,
\end{align}
\begin{equation}\label{total-mass}
\mathcal{N}_0 := \sum_{\lambda\in P_c} \Delta(\lambda)=\sum_{\mu\in \hat{P}_c} \hat{\Delta}  (\mu)=
\hat{\mathcal{N}}_0 .
\end{equation}

Furthermore, the total mass $\mathcal{N}_0$ of the {\em positive} discrete orthogonality measure $\Delta$ admits a compact
product representation of the form $\mathcal{N}_0=\text{Ind}(R) \mathcal{N}_c$ with
$\text{Ind}(R):=|P/Q|$ and $\mathcal{N}_c$
given by Tables \ref{tab-Nc-1} (for $R$ simply laced), \ref{tab-Nc-2} (for $R$ multiply laced with $\hat{R}=R$) and \ref{tab-Nc-3} (for $R$ multiply laced with $\hat{R}=R^\vee$).
\end{subequations}
\end{theorem}

\begin{table}
\begin{tabular}{|c|c|c|c|c|}
\hline
 $R$ &$\mathcal N_c $  & $h_{\text{g}}$ \\
\hline & &\\ [-4.7ex] \hline
$A_n$  &  ${\scriptstyle \prod_{k=1}^n (1+k\text{g} : \kappa_\varphi )_{c-1}}$   & ${\scriptstyle (n+1)\text{g}}$ \\
\hline
$D_n$     &  ${\scriptstyle (1+(n-1)\text{g} : \kappa_\varphi )_{c-1}\prod_{k=1}^{n-1} (1+(2k-1)\text{g} : \kappa_\varphi )_{c-1}}$  & ${\scriptstyle 2(n-1)\text{g}}$ \\
\hline
$E_6$     &  ${\scriptstyle (1+\text{g},1+4\text{g},1+5\text{g},1+7\text{g},1+8\text{g} , 1+11\text{g}  : \kappa_\varphi )_{c-1} } $ & ${\scriptstyle 12\text{g}}$\\
\hline
$E_7$     &
${\scriptstyle (1+\text{g},1+5\text{g},1+7\text{g},1+9\text{g},1+11\text{g} , 1+13\text{g}, 1+17\text{g}  : \kappa_\varphi )_{c-1}} $
 & ${\scriptstyle 18\text{g}}$\\ [1ex]
\hline
$E_8$     &
${\scriptstyle (1+\text{g},1+7\text{g},1+11\text{g},1+13\text{g},1+17\text{g} , 1+19\text{g}, 1+23\text{g}, 1+29\text{g}  : \kappa_\varphi )_{c-1} }$
 & ${\scriptstyle 30\text{g}}$ \\
\hline
\end{tabular}
\vspace{2ex}
\caption{Value of $\mathcal N_c=\frac{\mathcal{N}_0}{\text{Ind}(R)}$ when $R$ is simply laced. }
\protect{\label{tab-Nc-1}}
\end{table}

\begin{table}
\begin{tabular}{|c|c|c|c|c|}
\hline
 $R$ &$\mathcal N_c $  & $h_{\text{g}}$ \\
\hline & \\ [-4.7ex] \hline
$B_n$    &
$
{\scriptstyle (1+\text{g}_\vartheta+2(n-1)\text{g}_\varphi:\kappa_\vartheta)_{2c-1}\prod_{k=1}^{n-1}}
\frac{(1+k\text{g}_\varphi,1+\text{g}_\vartheta+(n+k-2)\text{g}_\varphi:\kappa_\varphi)_{c-1}}{(1+2k\text{g}_\varphi:\kappa_\vartheta)_{c-1}^2}
$
& ${\scriptstyle 2(n-1)\text{g}_\varphi+\text{g}_\vartheta}$  \\
\hline
$C_n$     &
$\begin{array}{c}
{\scriptstyle (1+(n-1)\text{g}_\vartheta+2\text{g}_\varphi:\kappa_\vartheta)_{2c-1} \prod_{k=1}^{n-2}(1+(n+k)\text{g}_\vartheta+2\text{g}_\varphi:\kappa_\vartheta)_c^2} \\
{\scriptstyle \times \prod_{k=0}^{n-1}}  \frac{(1+k\text{g}_\vartheta+\text{g}_\varphi:\kappa_\varphi)_{c-1}}{(1+2k\text{g}_\vartheta+2\text{g}_\varphi:\kappa_\vartheta)_{2c-1}}
\end{array}
$
 &${\scriptstyle 2\text{g}_\varphi+(n-1)\text{g}_\vartheta}$ \\
\hline
$F_4$     &
$\begin{array}{c}{\scriptstyle (1+\text{g}_\varphi,1+\text{g}_\vartheta+3\text{g}_\varphi,1+2\text{g}_\vartheta+3\text{g}_\varphi,1+3\text{g}_\vartheta+5\text{g}_\varphi:\kappa_\varphi)_{c-1}}  \\
{\scriptstyle \times } \frac{(1+3\text{g}_\vartheta+4\text{g}_\varphi:\kappa_\vartheta)_{c-1}^2(1+5\text{g}_\vartheta+6\text{g}_\varphi:\kappa_\vartheta)_{c}^2}
{(1+4\text{g}_\varphi,1+2\text{g}_\vartheta+6\text{g}_\varphi:\kappa_\vartheta)_{c-1}^2}
\end{array}
$
 & ${\scriptstyle 6\text{g}_\varphi+3\text{g}_\vartheta}$  \\
\hline
$G_2$     &
$
\frac
{(1+\text{g}_\varphi,1+\text{g}_\vartheta+2\text{g}_\varphi:\kappa_\varphi)_{c-1}(1+2\text{g}_\vartheta + 3\text{g}_\varphi :\kappa_\vartheta)_c^2  }
{(1+3\text{g}_\varphi:\kappa_\vartheta)_{c-1}^2}
$
& ${\scriptstyle 3\text{g}_\varphi+\text{g}_\vartheta}$ \\
\hline
\end{tabular}
\vspace{2ex}
\caption{Value of $\mathcal N_c=\frac{\mathcal{N}_0}{\text{Ind}(R)}$ when $R$ is multiply laced with $\hat{R}=R$. }
\protect{\label{tab-Nc-2}}
\end{table}

\begin{table}
\begin{tabular}{|c|c|c|c|c|}
\hline
 $R$ &$\mathcal N_c $  & $h_{\text{g}}$ \\
\hline & \\ [-4.7ex] \hline
$B_n$    &
$
\frac{\prod_{k=1}^{n} (1+\text{g}_\vartheta+(n+k-2)\text{g}_\varphi :\kappa )_{c-1}\prod_{k=1}^{[n/2]} (1+(2k-1)\text{g}_\varphi :\kappa )_{c-1}}{\prod_{k=1}^{[n/2]} (1+2(n-k)\text{g}_\varphi :\kappa )_{c-1}}
$
 & ${\scriptstyle 2(n-1)\text{g}_\varphi+2\text{g}_\vartheta}$ \\
\hline
$C_n$     &
$
\frac{\prod_{k=1}^{n} (1+(k-1)\text{g}_\vartheta+\text{g}_\varphi :\kappa )_{c-1}\prod_{k=1}^{[n/2]} (1+(2n-2k-1)\text{g}_\vartheta+2\text{g}_\varphi :\kappa )_{c-1}}{\prod_{k=1}^{[n/2]} (1+2(k-1)\text{g}_\vartheta+2\text{g}_\varphi :\kappa )_{c-1}}
$
& ${\scriptstyle 2\text{g}_\varphi+2(n-1)\text{g}_\vartheta }$ \\
\hline
$F_4$     &
$
\frac{(1+\text{g}_\varphi,1+\text{g}_\vartheta+3\text{g}_\varphi,1+2\text{g}_\vartheta+3\text{g}_\varphi,1+3\text{g}_\vartheta+4\text{g}_\varphi,1+3\text{g}_\vartheta+5\text{g}_\varphi,1+5\text{g}_\vartheta+6\text{g}_\varphi:\kappa)_{c-1}}
{(1+4\text{g}_\varphi,1+2\text{g}_\vartheta+6\text{g}_\varphi:\kappa)_{c-1}}
$
& ${\scriptstyle 6\text{g}_\varphi+6\text{g}_\vartheta }$  \\
\hline
$G_2$     &
$
\frac
{(1+\text{g}_\varphi, 1+\text{g}_\vartheta+2\text{g}_\varphi, 1+ 2\text{g}_\vartheta + 3\text{g}_\varphi  :\kappa)_{c-1}}
{(1+3\text{g}_\varphi :\kappa)_{c-1}}
$
& ${\scriptstyle 3\text{g}_\varphi+3\text{g}_\vartheta}$  \\
\hline
\end{tabular}
\vspace{2ex}
\caption{Value of $\mathcal N_c=\frac{\mathcal{N}_0}{\text{Ind}(R)}$ when $R$ is multiply laced with $\hat{R}=R^\vee$. }
\protect{\label{tab-Nc-3}}
\end{table}

\begin{remark}
For $R$ of type $A$ Theorem \ref{main:thm} amounts to \cite[\text{Eq.}~(4.15b)]{die-vin:quantum} whereas for $R=\hat{R}$ of type $C$ one recovers a special case of the orthogonality in \cite[\text{Sec.~6}]{die-sto:multivariable} (with $\text{g}=\text{g}_\vartheta$, $\text{g}_a=\text{g}_b=\text{g}_\varphi$ and $\text{g}_c=\text{g}_d=0$).
\end{remark}

\begin{remark}\label{tr:rem}
For $\kappa$ as in Theorem \ref{main:thm}, the Macdonald parameters $q$ and $t$ satisfy the
truncation relation
\begin{subequations}
\begin{equation}\label{tr}
t_\vartheta^{m\langle\rho_\vartheta,\hat{\psi}^\vee\rangle}
t_\varphi^{\langle\rho_{\varphi\setminus\vartheta},\hat{\psi}^\vee\rangle} t_{\psi}\, q_\varphi^c=1,
\end{equation}
where $m:=u_\varphi/u_\vartheta$ ($ \in \{1,2,3\} $) and
\begin{equation}\label{rho}
\rho:=\frac{1}{2}\sum_{\alpha\in R^+}\alpha ,\qquad \rho_\vartheta:=\frac{1}{2}\sum_{\substack{\alpha\in R^+\\ \|\alpha\|=\|\vartheta\|}}\!\!\alpha ,\qquad \rho_{\varphi\setminus\vartheta}:=
\frac{1}{2}\sum_{\substack{\alpha\in R^+\\ \|\alpha\|\neq \|\vartheta\|}}\!\!\!\!\!\!\alpha =\rho-\rho_{\vartheta}
\end{equation}
\end{subequations}
(so $\rho_{\text{g}}=\text{g}_\vartheta\rho_\vartheta+\text{g}_\varphi\rho_{\varphi\setminus\vartheta}$).
If the dual root system $R^\vee$ is isomorphic to $R$, the truncation relation in Eq. ~\eqref{tr} becomes of the form $t_\vartheta^{h/2}t_\varphi^{h/2}q_\varphi^c=1$, where
$h=h(R):=\langle\rho ,\vartheta^\vee\rangle +1$ (the Coxeter number of $R$).
In particular, for $R$ simply laced the truncation relation reads $t^{h}q_\varphi^c=1$.
\end{remark}

\begin{remark}\label{root-of-1:rem} Since for $\kappa$ as in Theorem \ref{main:thm} one has that
 $q_a=\exp (\frac{2\pi i}{m_a(h_{\text{g}}+c)})$ and  $t_a=\exp (\frac{2\pi i \text{g}_a }{m_a(h_{\text{g}}+c)})$
 with $m_a:=u_\varphi/u_a$ ($\in \{ 1,m\}$, cf. Remark \ref{tr:rem}) for $a\in R\cup\hat{R}$, it is clear that when is  $\text{g}$ integral-valued $q_a$ and $t_a$ are
roots of unity, cf. \cite[\text{Sec}.~5]{che:macdonalds} and \cite[\text{Sec}.~5]{kir:inner}.
\end{remark}

\begin{remark}
Both orthogonality relations in Eqs. \eqref{ort-rel1} and \eqref{ort-rel2} are equivalent to the unitarity of the matrix $[S_{\lambda ,\mu}]_{\lambda\in P_c,\mu\in\hat{P}_c}$ with
\begin{equation}
S_{\lambda ,\mu} :=  \left(\frac{\Delta (\lambda )\hat{\Delta} (\mu)}{\mathcal{N}_0}\right)^{1/2} P_\lambda (\hat{\rho}_{\text{g}}+\mu) .
\end{equation}
Since the parameter specialization in Eq. \eqref{special-parameter} preserves the duality symmetry in the sense that $\hat{\kappa}=\frac{\pi}{u_{\psi} (\hat{h}_{\text{g}}+c)}   =\kappa$
(because $\hat{h}_{\text{g}}=\langle \hat{\rho}_{\text{g}}, \hat{\varphi}^\vee\rangle +  \text{g}_\varphi=h_{\text{g}}$ and $u_\psi=u_\varphi$),
the matrix in question
inherits from Eq. \eqref{duality-symmetry} in addition the duality symmetry $\hat{S}_{\mu,\lambda}=S_{\lambda,\mu}$.
\end{remark}

\begin{remark}\label{equal-label:rem}
When $R$ is either simply laced or multiply laced with $\hat{R}=R^\vee$, the product formulas in Tables \ref{tab-Nc-1} and \ref{tab-Nc-3} follow from the terminating Aomoto-Ito-Macdonald-type basic hypergeometric summation formula in \cite[Thm.~3]{maz:finite}. In this situation the value of $\mathcal{N}_c$ can be rewritten as (cf. \cite[Eq.~(4.6)]{maz:finite})
\begin{equation}\label{NcRhat}
  \mathcal N_c
=\frac{\prod_{\alpha\in R^+} (1+\langle  \rho_{\text{g}},\alpha^\vee \rangle \! :\! \kappa_\alpha)_{c-1}}
       {\prod_{\alpha\in R^+\backslash I} (1+\langle  \rho_{\text{g}},\alpha^\vee \rangle  - \text{g}_\alpha \! : \! \kappa_\alpha)_{c-1}} ,
\end{equation}
where $I\subseteq R^+$ consists of the simple roots of $R$.
In the equal label case, i.e. with the root multiplicity function $\text{g}$ being constant (so in particular when $R$ is simply laced), the product formula in \eqref{NcRhat} simplifies to (cf. \cite[Eq.~(4.4)]{maz:finite})
\begin{equation}
  \mathcal N_c
=\prod_{k=1}^n (1+\text{g} e_k \!: \! \kappa_\varphi)_{c-1} ,
\end{equation}
where $e_1,\ldots ,e_n$ refer to the exponents of the Weyl group (thus explaining the structure of the formulas in Table \ref{tab-Nc-1} and those in Table \ref{tab-Nc-3} when $\text{g}_\vartheta=\text{g}_\varphi=\text{g}$).
\end{remark}

\begin{remark} For $\text{g}=1$, Macdonald's polynomials become Weyl characters \cite{mac:orthogonal}. Theorem \ref{main:thm} then boils down to the following elementary orthogonality relations for the antisymmetric monomials
\begin{equation}
\chi_\lambda:=\sum_{w\in W} \det (w) e^{w\lambda} \qquad (\lambda\in P^+)
\end{equation}
at $\kappa=\frac{\pi}{u_\varphi (\hbar+c)}$  with
$\hbar =\hbar (R,\hat{R}):=\langle \rho, \hat \psi^\vee\rangle +  1$ ($\in\{ h,h^\vee\}$),\footnote{Here $h^\vee=h^\vee(R):=\langle \rho,\varphi^\vee\rangle+1$ (the dual Coxeter number of $R$).} cf. \cite[\S 13.8]{kac:infinite}, \cite[\text{Sec.}~6]{kir:inner}, \cite[\text{Sec.}~5.3]{hri-pat:discretization}, \cite[\text{Sec}.~8.4]{die-ems:discrete} (and also \cite[\text{Sec}.~4.2]{die:finite-dimensional} and
\cite[\text{Sec}.~6.2]{kor-stro:slnk} for the special case when $R$ is of type $A$):
\begin{equation}\label{weyl-characters-orto}
\sum_{\lambda\in P_c} \chi_{\rho+\lambda} (\hat{\rho}+\mu )\overline{\chi_{\rho+\lambda} (\hat{\rho}+\tilde\mu )}=
\begin{cases}
0&\text{if}\ \mu \neq \tilde\mu \\
 \text{Ind}(R,\hat{R}) (\hbar+c)^n &\text{if}\ \mu =\tilde\mu
\end{cases}
\end{equation}
($\mu,\tilde{\mu}\in \hat{P}_c$).
Here the index governing the value of the quadratic norms is defined as $\text{Ind}(R,\hat{R}):=|P/(u_\varphi \hat Q^\vee) |= |P/Q|\, |Q/(u_\varphi \hat Q^\vee) |$ (so $\text{Ind}(R,\hat{R})=\text{Ind}(R)$ if $R$ is simply-laced or $\hat R=R^\vee$, and  $\text{Ind}(R,\hat{R})=m^{n_\vartheta}\text{Ind}(R) $ with $n_\vartheta$ denoting the number of short simple roots of $R$ otherwise).
The orthogonality in \eqref{weyl-characters-orto} readily follows  from the orthogonality of the discrete Fourier basis on $P/(\hbar+c)u_\varphi \hat Q^\vee$:
\begin{equation}\label{character:ort}
\sum_{\lambda\in P/(\hbar+c)u_\varphi \hat Q^\vee} e^{\frac{2\pi i}{u_\varphi (\hbar+c)} \langle \lambda,\mu\rangle } =
\begin{cases} 0 &\text{if}\ \mu\in \hat{P}\setminus (\hbar+c)u_\varphi Q^\vee \\
|P/(\hbar+c)u_\varphi \hat{Q}^\vee |&\text{if}\ \mu\in (\hbar+c)u_\varphi Q^\vee
\end{cases},
\end{equation}
 upon antisymmetrization and  using that $|P/(\hbar+c)u_\varphi \hat{Q}^\vee|=\text{Ind}(R,\hat{R})(\hbar+c)^n$.
 When comparing the values of the quadratic norms in \eqref{weyl-characters-orto} with the ones obtained from Theorem \ref{main:thm} through specialization, one deduces the following remarkable trigonometric identity for root systems
at $\text{g}=1$ (and thus $\kappa_\alpha=\frac{\pi}{m_\alpha (\hbar+c)}$):
\begin{equation}
\mathcal{N}_c\prod_{\alpha\in {R}^+} \sin \kappa_\alpha\langle {\rho},\alpha^{\vee}\rangle = \frac{\text{Ind}(R,\hat{R})}{\text{Ind}(R)}(\hbar+c)^n
\end{equation}
(cf. also Remark \ref{equal-label:rem}).
\end{remark}

\section{Analyticity and Orthogonality}\label{sec4}
From now on it is always assumed---unless explicitly stated otherwise---that $\kappa$ takes the value in Eq. \eqref{special-parameter} (with $\text{g}>0$).

\subsection{Meromorphy}
The regularity of (the expansion coefficients of) the Macdonald polynomials in Section \ref{sec3} at the above value of $\kappa$ hinges for generic $\text{g}>0$ on the following lemma.

\begin{lemma}[Nondegeneracy Eigenvalues]\label{nondegeneracy:lem}
For any $\lambda,\mu\in P_c$ \eqref{Pc} with $\lambda\neq\mu$, there exists a small weight $\omega\in\hat{P}^+$ such that the eigenvalues $E_\omega$ \eqref{spectrum} are distinct:
\begin{equation}
E_\omega(\rho_{\text{g}}+\lambda)\neq E_\omega(\rho_{\text{g}}+\mu)
\end{equation}
as analytic functions in $\text{g}>0$.
\end{lemma}
\begin{proof}
For the classical root systems all fundamental weights are small.
The stated nondegeneracy of the eigenvalues follows in this situation from the fact that the trigonometric polynomials $\hat{m}_\omega$ (and thus $E_\omega$) corresponding to the fundamental weights
$\omega\in \hat{P}^+$ separate the points of $\rho_{\text{g}}+P_c\subset \frac{\pi}{\kappa}\hat{A}$, where $\hat{A}$ refers to
the open fundamental alcove
$\{ x\in V\mid 0<\langle x,\alpha\rangle<1,\,\forall\alpha\in \hat{R}^+\}$ of the affine Weyl group
$W\ltimes \hat{Q}^\vee$. For the exceptional root systems the nondegeneracy in question follows in turn from a case by case examination of the relevant eigenvalues
that is performed in the appendix at the end of the paper.
\end{proof}

To infer the regularity of $p_\lambda$ for $\lambda\in P_c$ and generic values of the multiplicity parameter, we combine Lemma \ref{nondegeneracy:lem} with the expressions for the Macdonald difference operators in Eqs. \eqref{D}--\eqref{U} to conclude that the Macdonald polynomials at issue are meromorphic in $\text{g}>0$ as a consequence of an explicit representation similar to that in Eq. \eqref{P:explicit}.

\begin{proposition}[Meromorphy]\label{meromorphy:prp}
The Macdonald polynomials $p_\lambda$, $\lambda\in P_c$ are meromorphic in $\text{g}>0$.
\end{proposition}
\begin{proof}
For $\lambda\in P_c$ and $\mu <\lambda$, one has that $\mu\in P_c$ (because $\hat{\psi}\in P^+$). In this situation,
let $\omega_{\lambda\mu}$ denote a small weight in $\hat{P}^+$ from Lemma \ref{nondegeneracy:lem} such that $E_{\omega_{\lambda\mu}}(\rho_{\text{g}}+\lambda)\neq E_{\omega_{\lambda\mu}}(\rho_{\text{g}}+\mu)$ as analytic expressions in $\text{g}$. The meromorphy of $p_\lambda$ in $\text{g}>0$ is now immediate from the formula
(cf. Eq.~\eqref{P:explicit})
\begin{equation*}
p_\lambda=
\biggl(
\prod_{\substack{\mu\in P_c\\  \mu < \lambda}} \frac{\mathcal{D}_{\omega_{\lambda\mu }}-E_{\omega_{\lambda\mu}} (\rho_{\text{g}}+\mu)}{E_{\omega_{\lambda\mu}} (\rho_{\text{g}}+\lambda)-E_{\omega_{\lambda\mu}} (\rho_{\text{g}}+\mu)}
\biggr)
m_\lambda ,
\end{equation*}
combined with the explicit expression for $\mathcal{D}_{\omega_{\lambda\mu }}$ stemming from Eqs. \eqref{D}--\eqref{U}.
\end{proof}

\subsection{Finite Macdonald difference operators}
For the parameter regime in Theorem \ref{main:thm} (all factors of)  $\Delta (\lambda)$ ($\lambda\in P_c$) and $\hat{\Delta}(\mu)$ ($\mu\in\hat{P}_c$)
are positive because the arguments of the sine functions take values in the interval $(0,\pi)$, as is readily seen from the following estimates.
\begin{lemma}[Moment Bounds]\label{mb:lem}
For any $\lambda\in P_c$ and $\alpha\in R^+$, the following inequalities hold:
\begin{equation}
(i)\ \langle \lambda,\alpha^\vee\rangle\leq m_\alpha\langle \lambda,\hat{\psi}^\vee\rangle\leq m_\alpha c\quad\text{and}\quad (ii)\ \text{g}_\alpha\leq\langle \rho_{\text{g}},\alpha^\vee\rangle\leq m_\alpha h_{\text{g}} -\text{g}_\alpha ,
\end{equation}
i.e.  $0<\text{g}_\alpha\leq \langle\rho_{\text{g}}+\lambda,\alpha^\vee\rangle\leq m_\alpha (h_{\text{g}}+c)-\text{g}_\alpha<m_\alpha (h_{\text{g}}+c)$
(where $m_\alpha=u_\varphi/u_\alpha$, cf. Remark \ref{root-of-1:rem}).
\end{lemma}
\begin{proof}
Part $(i)$ is straightforward: $\langle \lambda,\alpha^\vee\rangle=u_\alpha^{-1}\langle \lambda,\hat{\alpha}\rangle\leq
u_\alpha^{-1}\langle \lambda,\psi\rangle=m_\alpha\langle \lambda,\hat{\psi}^\vee\rangle\leq m_\alpha c$.
For the proof of part $(ii)$ we write
$\rho_{\text{g}}=\text{g}_\vartheta\rho_\vartheta+\text{g}_\varphi\rho_{\varphi\setminus\vartheta}$ (cf. Remark \ref{tr:rem}), i.e.
$\langle \rho_{\text{g}},\alpha^\vee\rangle=\text{g}_\vartheta \langle \rho_\vartheta,\alpha^\vee\rangle + \text{g}_\varphi\langle \rho_{\varphi\setminus\vartheta},\alpha^\vee\rangle$. The lower bound $\text{g}_\alpha$ is now immediate from the fact that
$\langle \rho_\vartheta,\alpha^\vee\rangle>0$ if $\|\alpha\|=\|\vartheta\|$ and $\langle \rho_{\varphi\setminus\vartheta},\alpha^\vee\rangle>0$ if $\|\alpha\|\neq \|\vartheta\|$, because the (possibly empty) parabolic subsystems of $R$ corresponding to the stabilizers of $\rho_\vartheta$ and $\rho_{\varphi\setminus\vartheta}$ are generated  by the simple roots $\beta$
with $\|\beta \|\neq \|\vartheta\|$ and by the simple roots $\beta$ with $\|\beta \|=\|\vartheta\|$, respectively. To infer the upper bound, we compute:
\begin{align*}
m_\alpha h_{\text{g}}&-\langle \rho_{\text{g}},\alpha^\vee\rangle-\text{g}_\alpha \\
&= \text{g}_\vartheta\langle \rho_\vartheta,m_\alpha\hat{\psi}^\vee-\alpha^\vee\rangle+
\text{g}_\varphi\langle \rho_{\varphi\setminus\vartheta},m_\alpha \hat{\psi}^\vee-\alpha^\vee\rangle+m_\alpha \text{g}_\psi-\text{g}_\alpha \\
&=
\begin{cases}
\text{g}_\vartheta(\langle \rho_\vartheta,\vartheta^\vee-\alpha^\vee\rangle+1)+
\text{g}_\varphi\langle \rho_{\varphi\setminus\vartheta},\vartheta^\vee-\alpha^\vee\rangle-\text{g}_\alpha &\text{if}\ \hat{R}=R^\vee ,\\
\text{g}_\vartheta\langle \rho_\vartheta,m_\alpha\varphi^\vee-\alpha^\vee\rangle+
\text{g}_\varphi(\langle \rho_{\varphi\setminus\vartheta},m_\alpha \varphi^\vee-\alpha^\vee\rangle+m_\alpha ) -\text{g}_\alpha &\text{if}\ \hat{R}=R .
\end{cases}
\end{align*}
If $\hat{R}=R^\vee$ the nonnegativity of the expression on the RHS is manifest when $\|\alpha\|=\|\vartheta\|$, while for
$\|\alpha\|\neq \|\vartheta\|$
the nonnegativity follows from
the fact that $\langle \rho_{\varphi\setminus\vartheta},\vartheta^\vee-\alpha^\vee\rangle\geq \langle \rho_{\varphi\setminus\vartheta},\vartheta^\vee-\varphi^\vee\rangle>0$ (as for $R$ multiply laced the
decomposition $\varphi-\vartheta$ in the simple basis contains a simple root $\beta$ with $\|\beta\|=\|\vartheta\|$).
Similarly, if $\hat{R}=R$  the nonnegativity of the RHS is manifest when $\|\alpha\|=\|\varphi\|$ (so $m_\alpha=1$), while
for $\|\alpha\|\neq \|\varphi\|$ the nonnegativity in question follows from the estimates
$\langle \rho_\vartheta,m_\alpha \varphi^\vee-\alpha^\vee\rangle = u_\alpha^{-1}\langle \rho_\vartheta, \varphi-\alpha\rangle\geq u_\alpha^{-1}\langle \rho_\vartheta, \varphi-\vartheta\rangle
>0$ and $\langle \rho_{\varphi\setminus\vartheta},m_\alpha \varphi^\vee-\alpha^\vee\rangle=
u_\alpha^{-1}\langle \rho_{\varphi\setminus\vartheta}, \varphi-\alpha\rangle
\geq 0$.
\end{proof}

Let $\ell^2(\hat{\rho}_{\text{g}}+\hat{P}_c,\hat{\Delta})$ denote the finite-dimensional Hilbert space of functions $f:\hat{\rho}_{\text{g}}+\hat{P}_c\to\mathbb{C}$ endowed with the inner product
\begin{equation}
\langle f,g\rangle_{\hat{\Delta}}:=\sum_{\mu\in\hat{P}_c} f(\hat{\rho}_{\text{g}}+\mu)\overline{g(\hat{\rho}_{\text{g}}+\mu)}\hat{\Delta} (\mu)\qquad
(f,g\in \ell^2(\hat{\rho}_{\text{g}}+\hat{P}_c,\hat{\Delta})).
\end{equation}
For $\omega\in\hat{P}^+$ small, we consider the following finite analog of the Macdonald difference operator $\mathcal{D}_\omega$ \eqref{D}--\eqref{U}
in the Hilbert space $\ell^2(\hat{\rho}_{\text{g}}+\hat{P}_c,\hat{\Delta})$:
\begin{equation}\label{D-finite}
(D_\omega f)(\hat{\rho}_{\text{g}}+\mu)=\sum_{\substack{\nu\in  \hat{P}(\omega)\\ \mu+\nu\in\hat{P}_c   }}
\sideset{}{'}\sum_{\eta\in W_\nu(w_\nu^{-1}\omega)}
V_\nu(\hat{\rho}_{\text{g}}+\mu)U_{\nu,\eta}(\hat{\rho}_{\text{g}}+\mu) f(\hat{\rho}_{\text{g}}+\mu+\nu)
\end{equation}
($f\in \ell^2(\hat{\rho}_{\text{g}}+\hat{P}_c,\hat{\Delta})$), where the prime indicates that the second sum is restricted to those $\eta\in W_\nu(w_\nu^{-1}\omega)$ for which the denominator of $U_{\nu,\eta}(\hat{\rho}_{\text{g}}+\mu) $ does not vanish.
The operator $D_\omega$ is well-defined because of the following lemma.

\begin{lemma}[Regularity of $V$]\label{regular:lem} For any $\nu\in \hat{P}(\omega)$  with
$\omega\in\hat{P}^+$ small,
the denominator of the coefficient $V_\nu(x)$ \eqref{V}
is nonzero at $x=\hat{\rho}_{\text{g}}+\mu$ for all $\mu\in \hat{P}_c$ such that $\mu+\nu\in\hat{P}_c$.
\end{lemma}
\begin{proof}
The denominator of the coefficient $V_\nu(\hat{\rho}_{\text{g}}+\mu)$  is built of factors of the form  $\sin\kappa_\alpha (\langle \hat{\rho}_{\text{g}}+\mu,\alpha^\vee\rangle)$ and
$\sin\kappa_\alpha (\langle \hat{\rho}_{\text{g}}+\mu,\alpha^\vee\rangle +1)$ with $\alpha\in\hat{R}$ and $\langle \nu,\alpha^\vee\rangle >0$. These factors can only become zero when $(i)$
$\kappa_\alpha \langle \hat{\rho}_{\text{g}}+\mu,\alpha^\vee\rangle\in \pi\mathbb{Z}$, i.e.
$\langle \hat{\rho}_{\text{g}}+\mu,\alpha^\vee\rangle\in m_\alpha(h_{\text{g}}+c)\mathbb{Z}$,
or when $(ii)$
$\kappa_\alpha (\langle \hat{\rho}_{\text{g}}+\mu,\alpha^\vee\rangle+1)\in \pi\mathbb{Z}$, i.e.
$\langle \hat{\rho}_{\text{g}}+\mu,\alpha^\vee\rangle+1\in m_\alpha(h_{\text{g}}+c)\mathbb{Z}$, respectively.
Since for any $\alpha\in \hat{R}$ and $\mu\in\hat{P}_c$:  $0<|\langle \hat{\rho}_{\text{g}},\alpha^\vee\rangle|<m_\alpha h_{\text{g}}$ and
$0\leq | \langle \mu,\alpha^\vee\rangle |\leq m_\alpha c$ (cf. Lemma \ref{mb:lem}), the zeros of type $(i)$ do not occur
as $0<|\langle \hat{\rho}_{\text{g}}+\mu,\alpha^\vee\rangle|<m_\alpha(h_{\text{g}}+c)$ (so the absolute value of the argument of the sine function in the factor $\sin\kappa_\alpha (\langle \hat{\rho}_{\text{g}}+\mu,\alpha^\vee\rangle)$ stays between $0$ and $\pi$).
When $\alpha\in \hat{R}^+$, the estimates in Lemma \ref{mb:lem} reveal that a zero of type $(ii)$ can only occur if
$\langle \hat{\rho}_{\text{g}}+\mu,\alpha^\vee\rangle+1= m_\alpha(h_{\text{g}}+c)$, i.e.
$\langle \mu,\alpha^\vee\rangle-m_\alpha c+1=m_\alpha h_{\text{g}}-\langle \hat{\rho}_{\text{g}},\alpha^\vee\rangle $ ($>0$), which implies that
$\langle \mu,\alpha^\vee\rangle=m_\alpha c$ and $\langle \hat{\rho}_{\text{g}},\alpha^\vee\rangle =m_\alpha h_{\text{g}}-1$. But then
$\langle \mu+\nu,\alpha^\vee\rangle> m_\alpha c$, i.e. $\mu+\nu\not\in\hat{P}_c$ (cf. part $(i)$ of Lemma \ref{mb:lem}).
Similarly, when $-\alpha\in \hat{R}^+$ a zero of type $(ii)$ can only occur if
$\langle \hat{\rho}_{\text{g}}+\mu,\alpha^\vee\rangle+1= 0$, i.e.
$\langle \mu,\alpha^\vee\rangle+1=-\langle \hat{\rho}_{\text{g}},\alpha^\vee\rangle $ ($>0$), which implies that
$\langle \mu,\alpha^\vee\rangle=0$ and $\langle \hat{\rho}_{\text{g}},\alpha^\vee\rangle =-1$. But then
$\langle \mu+\nu,-\alpha^\vee\rangle< 0$, i.e. $\mu+\nu\not\in\hat{P}_c$.
\end{proof}

The next lemma provides an explicit criterion characterizing which weights $\eta\in W_\nu(w_\nu^{-1}\omega)$
are omitted in the second  summation of Eq. \eqref{D-finite}.
It shows in particular that for generic $\text{g}$, or when $\mu\in \hat{P}_c$ is regular with respect to the action of the affine Weyl group $W\ltimes (cu_\varphi Q^\vee)$, the sum in question is simply over the full orbit
$W_\nu(w_\nu^{-1}\omega)$.

\begin{lemma}[Singularities of $U$]\label{denom-U:lem}
For any $\nu\in \hat{P}(\omega)$  and $\eta\in W_\nu(w_\nu^{-1}\omega)$ with
$\omega\in\hat{P}^+$ small,
the denominator of  $U_{\nu,\eta}(x)$ \eqref{U}
is zero at $x=\hat{\rho}_{\text{g}}+\mu$, $\mu\in \hat{P}_c$ iff there exists an $\alpha\in \hat{R}_\nu$ with $\langle \eta,\alpha^\vee\rangle=2$ such that
$(i)$ $\langle \mu,\alpha^\vee\rangle=0$ and $\langle \hat{\rho}_{\text{g}},\alpha^\vee\rangle =-1$,
{\em or}
 $(ii)$  $\langle \mu,\alpha^\vee\rangle=m_\alpha c$ and $\langle \hat{\rho}_{\text{g}},\alpha^\vee\rangle =m_\alpha h_{\text{g}}-1$
 (so in both cases $\mu+\eta\not\in P_c$).
\end{lemma}
\begin{proof}
The proof goes along the same lines as that of Lemma \ref{regular:lem}, upon replacing $\nu$ by $\eta$ and $\hat{R}$ by $\hat{R}_\nu$.
\end{proof}

By Lemma \ref{regular:lem}, the denominator of  $V_\nu(x)$, $\nu\in\hat{P}(\omega)$ can only become zero at
$x=\hat{\rho}_{\text{g}}+\mu$, $\mu\in \hat{P}_c$  if  $\mu+\nu\not\in\hat{P}_c$.
For such $\mu$ and $\nu$, however, the numerator of the coefficient at issue turns out to vanish identically.

\begin{lemma}[Vanishing Boundary Terms]\label{bc:lem} Let $\nu\in \hat{P}(\omega )$  with
$\omega\in\hat{P}^+$ small, and let $\mu\in P_c$. Then
the numerator of the coefficient $V_\nu(x)$ \eqref{V}
vanishes at $x=\hat{\rho}_{\text{g}}+\mu$ iff $\mu+\nu\not \in\hat{P}_c$.
\end{lemma}
\begin{proof} The proof is similar to that of Lemma \ref{regular:lem}.
The numerator of the coefficient $V_\nu(\hat{\rho}_{\text{g}}+\mu)$ consists of factors of the form  $\sin\kappa_\alpha (\langle \hat{\rho}_{\text{g}}+\mu,\alpha^\vee\rangle+\text{g}_\alpha)$ with $\alpha\in\hat{R}$ and $\langle \nu,\alpha^\vee\rangle >0$ and
$\sin\kappa_\alpha (\langle \hat{\rho}_{\text{g}}+\mu,\alpha^\vee\rangle +1+\text{g}_\alpha )$ with $\alpha\in\hat{R}$ and $\langle \nu,\alpha^\vee\rangle =2$. These factors  only become zero when $(i)$
$\kappa_\alpha (\langle \hat{\rho}_{\text{g}}+\mu,\alpha^\vee\rangle+\text{g}_\alpha)\in \pi\mathbb{Z}$, i.e.
$\langle \hat{\rho}_{\text{g}}+\mu,\alpha^\vee\rangle+\text{g}_\alpha\in m_\alpha(h_{\text{g}}+c)\mathbb{Z}$,
or when $(ii)$
$\kappa_\alpha (\langle \hat{\rho}_{\text{g}}+\mu,\alpha^\vee\rangle+1+\text{g}_\alpha )\in \pi\mathbb{Z}$, i.e.
$\langle \hat{\rho}_{\text{g}}+\mu,\alpha^\vee\rangle+1+\text{g}_\alpha\in m_\alpha(h_{\text{g}}+c)\mathbb{Z}$.
Upon invoking of Lemma \ref{mb:lem} it is seen that when $\alpha\in \hat{R}^+$, the zeros in question can only occur $(i)$ if
$\langle \hat{\rho}_{\text{g}}+\mu,\alpha^\vee\rangle+\text{g}_\alpha= m_\alpha(h_{\text{g}}+c)$, so $\langle \mu,\alpha^\vee\rangle\geq m_\alpha c$, or  $(ii)$ if
$\langle \hat{\rho}_{\text{g}}+\mu,\alpha^\vee\rangle+1+\text{g}_\alpha= m_\alpha(h_{\text{g}}+c)$, so $\langle \mu,\alpha^\vee\rangle\geq m_\alpha c-1$. In both cases this implies that $ \langle \mu+\nu,\alpha^\vee\rangle>m_\alpha c$, whence $\mu+\nu\not\in\hat{P}_c$.
Similarly, when $-\alpha\in \hat{R}^+$ the zeros in question  can only occur $(i)$ if
$\langle \hat{\rho}_{\text{g}}+\mu,\alpha^\vee\rangle+\text{g}_\alpha= 0$,  so
$\langle \mu,\alpha^\vee\rangle \geq 0$, or
 $(ii)$ if
$\langle \hat{\rho}_{\text{g}}+\mu,\alpha^\vee\rangle+1+\text{g}_\alpha= 0$,  so
$\langle \mu,\alpha^\vee\rangle \geq -1$.
In both cases this implies that $ \langle \mu+\nu,-\alpha^\vee\rangle<0$, whence $\mu+\nu\not\in\hat{P}_c$.
Reversely, if $\mu+\nu\not \in\hat{P}_c$ then either  $(i)$  $\langle \mu+\nu,\beta^\vee\rangle<0$ for some {\em simple} root $\beta\in \hat{R}^+$ or $(ii)$ $\langle \mu+\nu,\hat{\varphi}^\vee\rangle>c$. Since $\mu\in \hat{P}_c$ and
$|\langle \nu,\alpha^\vee\rangle|\leq 2$ for all $\alpha\in\hat{R}$, in either case we are in one of two situations:
$(ia)$ $\langle \mu,\beta^\vee\rangle=0$ with $\langle \nu,\beta^\vee\rangle<0$
or  $(ib)$ $\langle \mu,\beta^\vee\rangle=1$ with $\langle \nu,\beta^\vee\rangle=-2$, and
$(iia)$ $\langle \mu,\hat{\varphi}^\vee\rangle=c$ with $\langle \nu,\hat{\varphi}^\vee\rangle>0$
or  $(iib)$ $\langle \mu,\hat{\varphi}^\vee\rangle=c-1$ with $\langle \nu,\hat{\varphi}^\vee\rangle=2$.
In each situation
the numerator of $V_\nu(x)$
picks up a zero at $x=\hat{\rho}_{\text{g}}+\mu$ from the factor
$(ia)$ $\sin\kappa_\alpha (\langle \hat{\rho}_{\text{g}}+\mu,\alpha^\vee\rangle+\text{g}_\alpha)$ or
$(ib)$ $\sin\kappa_\alpha (\langle \hat{\rho}_{\text{g}}+\mu,\alpha^\vee\rangle +1+\text{g}_\alpha )$, where $\alpha=-\beta$,
and
$(iia)$ $\sin\kappa_{\hat{\varphi}} (\langle \hat{\rho}_{\text{g}}+\mu,\hat{\varphi}^\vee\rangle+\text{g}_{\hat{\varphi}})$ or
$(iib)$ $\sin\kappa_{\hat{\varphi}} (\langle \hat{\rho}_{\text{g}}+\mu,\hat{\varphi}^\vee\rangle +1+\text{g}_{\hat{\varphi}} )$.
Indeed, the arguments of these sine functions either $(i)$ vanish (since $\langle \hat{\rho}_{\text{g}},\beta^\vee\rangle=\text{g}_\beta$ for $\beta\in\hat{R}^+$ simple) or $(ii)$ they are equal to $\pi$ (since $\langle \hat{\rho}_{\text{g}},\hat{\varphi}^\vee\rangle+\text{g}_{\hat{\varphi}} =\hat{h}_{\text{g}}=h_{\text{g}}$).
\end{proof}

The numerator of $U_{\nu ,\eta}(x)$ enjoys an analogous vanishing property at the points $x=\hat{\rho}_g+\mu$, $\mu\in\hat{P}_c$ for which the denominator might become zero (where it is assumed that $\mu+\nu\in\hat{P}_c$ in view of Lemma \ref{bc:lem}).

\begin{lemma}[Vanishing of $U$]\label{num-U:lem}
Let $\nu\in \hat{P}(\omega)$  and $\eta\in W_\nu(w_\nu^{-1}\omega)$ with
$\omega\in\hat{P}^+$ small, and let $\mu\in \hat{P}_c$ with $\mu+\nu\in \hat{P}_c$.
Then the numerator of  $U_{\nu,\eta}(x)$ \eqref{U}
vanishes at $x=\hat{\rho}_{\text{g}}+\mu$ if there exists an $\alpha\in \hat{R}_\nu$ with $\langle \eta,\alpha^\vee\rangle=2$ such that
$(i)$ $\langle \mu,\alpha^\vee\rangle=0$ and $\alpha\not\in \hat{R}^+$,
{\em or}
 $(ii)$  $\langle \mu,\alpha^\vee\rangle=m_\alpha c$.
\end{lemma}
\begin{proof}
In the first case $(i)$, let  $-\beta\in \hat{R}^+$ be any simple root in the decomposition of $\alpha$ with respect to the simple basis of $\hat{R}$ for which
$\langle \eta ,\beta^\vee\rangle>0$. Since $\alpha$ belongs to the parabolic subsystem  $\hat{R}_\mu\cap \hat{R}_{\nu}= \hat{R}_\mu\cap \hat{R}_{\mu+\nu}$ with $\mu$ and $\mu+\nu$ dominant, the same is true for $\beta$.
Hence, the numerator of $U_{\nu,\eta}(x)$
picks up a zero at $x=\hat{\rho}_{\text{g}}+\mu$ from the factor
$\sin\kappa_\beta (\langle \hat{\rho}_{\text{g}}+\mu,\beta^\vee\rangle+\text{g}_\beta)$ (using that $\langle \hat{\rho}_{\text{g}},\beta^\vee\rangle=-\text{g}_\beta$ when $-\beta$ is simple).
In the second case $(ii)$, it follows from the first inequality in Lemma \ref{mb:lem} that
$\langle \mu,\hat{\varphi}^\vee\rangle=c$ and  $\langle \mu+\nu,\hat{\varphi}^\vee\rangle=c$, so
$\langle \nu,\hat{\varphi}^\vee\rangle=0$.
Since $\langle \eta ,\hat{\varphi}^\vee-\alpha^\vee\rangle=\langle \eta ,\hat{\varphi}^\vee\rangle-2\leq 0$,
we either have that $(iia)$ $\langle \eta ,\hat{\varphi}^\vee\rangle= 2$ or
$(iib)$ there exists a simple root $-\beta\in \hat{R}^+$ such that $-\hat{\beta}$ is contained in
the decomposition of $\varphi-\hat{\alpha}\in Q^+$ with respect to the simple basis of $R$ and $\langle \eta,\beta^\vee\rangle >0$.
Clearly $\varphi-\hat{\alpha}$
belongs to the root lattice of the parabolic subsystem
$R_\mu\cap R_{\mu+\nu}$, whence $\hat{\beta}\in R_\mu\cap R_{\mu+\nu}$, i.e.
$\beta\in \hat{R}_\nu\cap \hat{R}_\mu$.
The upshot is that in the former case $(iia)$ the
numerator of $U_{\nu,\eta}(x)$
picks up a zero at $x=\hat{\rho}_{\text{g}}+\mu$ from the factor
$\sin\kappa_{\hat{\varphi}} (\langle \hat{\rho}_{\text{g}}+\mu,\hat{\varphi}^\vee\rangle+\text{g}_{\hat{\varphi}})$, whereas in the latter case $(iib)$ we pick up a zero from the factor $\sin\kappa_\beta (\langle \hat{\rho}_{\text{g}}+\mu,\beta^\vee\rangle+\text{g}_\beta)$ (using again that $\langle \hat{\rho}_{\text{g}},\beta^\vee\rangle=-\text{g}_\beta$).
\end{proof}

We will refer to $\text{g}$ being {\em regular} if
\begin{equation}\label{regular-g}
\langle \hat{\rho}_{\text{g}},\alpha^\vee\rangle\not\in \{ 1,m_\alpha h_{\text{g}}-1\}\quad\text{for all}\quad \alpha\in \hat{R}^+
\end{equation}
 (which is the case generically). The
above analysis reveals that when $\text{g}$ is regular
the denominators of $V_\nu(x)$ \eqref{V}  and  $U_{\nu,\eta}(x) $ \eqref{U} do not have zeros
at $x=\hat{\rho}_{\text{g}}+\mu$ for all $\mu\in \hat{P}_c$, while the vanishing of the numerators---at places where the denominators might get annihilated if the root multiplier fails to be regular---does persist independent of whether $\text{g}$ is regular or not.
In other words, for regular root multipliers there are no poles coalescing with the zeros in the numerators and we arrive in this situation at the finite  operator $D_\omega$  \eqref{D-finite} by restricting the action
of the Macdonald difference operator $\mathcal{D}_\omega$ in Eqs. \eqref{D}--\eqref{U} to functions
$f$ supported on $\hat{\rho}_{\text{g}}+\hat{P}_c$. In general, the finite operator $D_\omega$ \eqref{D-finite} is to be viewed as the continuation of this restriction of the Macdonald operator $\mathcal{D}_\omega$ from regular $\text{g}$ to arbitrary $\text{g}>0$.

We will next determine the adjoint of  $D_\omega$ in the Hilbert space $\ell^2(\hat{\rho}_{\text{g}}+\hat{P}_c,\hat{\Delta})$. The computation hinges on the following elementary recurrence relation for the orthogonality measure $\hat{\Delta}$ in terms of the coefficients $V_\nu$.

\begin{lemma}[Recurrence Relation]\label{recurrence:lem}
Let $\omega\in\hat{P}^+$ be small and  let $\mu\in\hat{P}_c$.
Then for any $\nu\in \hat{P}(\omega)$ such that $\mu+\nu\in\hat{P}_c$, one has that
\begin{align}
\hat{\Delta}(\mu+\nu) V_{-\nu}(\hat{\rho}_{\text{g}}+\mu+\nu)
=\hat{\Delta}(\mu) V_\nu(\hat{\rho}_{\text{g}}+\mu) .
\end{align}
\end{lemma}
\begin{proof}
Multiplication of
\begin{align*}
\hat{\Delta}(\mu+\nu)
&=\prod_{\alpha\in \hat{R}^+}
 \frac{\sin\kappa_\alpha \langle \hat{\rho}_{\text{g}}+\mu+\nu,\alpha^\vee\rangle}{\sin\kappa_\alpha \langle \hat{\rho}_{\text{g}},\alpha^\vee\rangle}
 \frac{(\langle \hat{\rho}_{\text{g}},\alpha^\vee \rangle+\text{g}_\alpha\! : \!\kappa_\alpha)_{\langle\mu+\nu,\alpha^\vee\rangle}}
         {(\langle \hat{\rho}_{\text{g}},\alpha^\vee \rangle+1-\text{g}_\alpha\! : \!\kappa_\alpha)_{\langle \mu+\nu,\alpha^\vee\rangle}}\\
&=
\prod_{\alpha\in \hat{R}^+}
 \frac{\sin\kappa_\alpha \langle \hat{\rho}_{\text{g}}+\mu,\alpha^\vee\rangle}{\sin\kappa_\alpha \langle \hat{\rho}_{\text{g}},\alpha^\vee\rangle}
 \frac{(\langle \hat{\rho}_{\text{g}},\alpha^\vee \rangle+\text{g}_\alpha\! : \!\kappa_\alpha)_{\langle\mu,\alpha^\vee\rangle}}
         {(\langle \hat{\rho}_{\text{g}},\alpha^\vee \rangle+1-\text{g}_\alpha\! : \!\kappa_\alpha)_{\langle \mu,\alpha^\vee\rangle}}\\
&\quad \times  \prod_{\substack{\alpha\in \hat{R}^+\\ \langle \nu,\alpha^\vee \rangle>0  }}
\frac{\sin \kappa_\alpha (\langle \hat{\rho}_{\text{g}}+\mu,\alpha^\vee \rangle+1)}
{\sin\kappa_\alpha   \langle \hat{\rho}_{\text{g}}+\mu,\alpha^\vee \rangle      }
 \frac{\sin \kappa_\alpha (\langle \hat{\rho}_{\text{g}}+\mu,\alpha^\vee \rangle+\text{g}_\alpha)}
{\sin\kappa_\alpha    (\langle \hat{\rho}_{\text{g}}+\mu,\alpha^\vee \rangle  +1-\text{g}_\alpha)     }
 \\
&\quad \times  \prod_{\substack{\alpha\in \hat{R}^+\\ \langle \nu,\alpha^\vee \rangle=2  }}
\frac{\sin \kappa_\alpha (\langle \hat{\rho}_{\text{g}}+\mu,\alpha^\vee \rangle+2)}
{\sin\kappa_\alpha   (\langle \hat{\rho}_{\text{g}}+\mu,\alpha^\vee \rangle     +1) }
\frac{\sin \kappa_\alpha (\langle \hat{\rho}_{\text{g}}+\mu,\alpha^\vee \rangle+1+\text{g}_\alpha)}
{\sin\kappa_\alpha    (\langle \hat{\rho}_{\text{g}}+\mu,\alpha^\vee \rangle  +2-\text{g}_\alpha)     }
 \\
&\quad \times  \prod_{\substack{\alpha\in \hat{R}^+\\ \langle \nu,\alpha^\vee \rangle<0  }}
\frac{\sin \kappa_\alpha (\langle \hat{\rho}_{\text{g}}+\mu,\alpha^\vee \rangle-1)}
{\sin\kappa_\alpha   \langle \hat{\rho}_{\text{g}}+\mu,\alpha^\vee \rangle      }
 \frac{\sin \kappa_\alpha (\langle \hat{\rho}_{\text{g}}+\mu,\alpha^\vee \rangle-\text{g}_\alpha)}
{\sin\kappa_\alpha    (\langle \hat{\rho}_{\text{g}}+\mu,\alpha^\vee \rangle  -1+\text{g}_\alpha)     }
\\
&\quad \times  \prod_{\substack{\alpha\in \hat{R}^+\\ \langle \nu,\alpha^\vee \rangle=-2  }}
\frac{\sin \kappa_\alpha (\langle \hat{\rho}_{\text{g}}+\mu,\alpha^\vee \rangle-2)}
{\sin\kappa_\alpha   (\langle \hat{\rho}_{\text{g}}+\mu,\alpha^\vee \rangle     -1) }
\frac{\sin \kappa_\alpha (\langle \hat{\rho}_{\text{g}}+\mu,\alpha^\vee \rangle-1-\text{g}_\alpha)}
{\sin\kappa_\alpha    (\langle \hat{\rho}_{\text{g}}+\mu,\alpha^\vee \rangle  -2+\text{g}_\alpha)     }
\end{align*}
by
\begin{multline*}
V_{-\nu}(\hat{\rho}_{\text{g}}+\mu+\nu)=V_{\nu}(-\hat{\rho}_{\text{g}}-\mu-\nu)=\\
  \prod_{\substack{\alpha\in \hat{R}^+\\     \langle \nu,\alpha^\vee\rangle>0 }}
\frac{\sin\kappa_\alpha(\langle \hat{\rho}_{\text{g}}+\mu,\alpha^\vee \rangle+1-\text{g}_\alpha)}{\sin\kappa_\alpha(\langle \hat{\rho}_{\text{g}}+\mu,\alpha^\vee \rangle +1)}
  \prod_{\substack{\alpha\in \hat{R}^+\\     \langle \nu,\alpha^\vee\rangle=2 }}
     \frac{\sin\kappa_\alpha(\langle \hat{\rho}_{\text{g}}+\mu,\alpha^\vee \rangle+2-\text{g}_\alpha)}
             {\sin\kappa_\alpha(\langle \hat{\rho}_{\text{g}}+\mu,\alpha^\vee \rangle+2)}\\
\times \prod_{\substack{\alpha\in \hat{R}^+\\     \langle \nu,\alpha^\vee\rangle<0 }}
   \frac{\sin\kappa_\alpha(\langle \hat{\rho}_{\text{g}}+\mu,\alpha^\vee \rangle-1+\text{g}_\alpha)}
           {\sin {\kappa_\alpha}(\langle \hat{\rho}_{\text{g}}+\mu,\alpha^\vee \rangle -1)}
\prod_{\substack{\alpha\in \hat{R}^+\\     \langle \nu,\alpha^\vee\rangle=-2 }}
   \frac{\sin\kappa_\alpha(\langle \hat{\rho}_{\text{g}}+\mu,\alpha^\vee \rangle-2+\text{g}_\alpha)}
           {\sin\kappa_\alpha(\langle \hat{\rho}_{\text{g}}+\mu,\alpha^\vee \rangle -2)}
\end{multline*}
entails
\begin{align*}
&\prod_{\alpha\in \hat{R}^+}
 \frac{\sin\kappa_\alpha \langle \hat{\rho}_{\text{g}}+\mu,\alpha^\vee\rangle}{\sin\kappa_\alpha \langle \hat{\rho}_{\text{g}},\alpha^\vee\rangle}
 \frac{(\langle \hat{\rho}_{\text{g}},\alpha^\vee \rangle+\text{g}_\alpha\! : \!\kappa_\alpha)_{\langle\mu,\alpha^\vee\rangle}}
         {(\langle \hat{\rho}_{\text{g}},\alpha^\vee \rangle+1-\text{g}_\alpha\! : \!\kappa_\alpha)_{\langle \mu,\alpha^\vee\rangle}}\\
&
\times
 \prod_{\substack{\alpha\in \hat{R}^+\\ \langle \nu,\alpha^\vee\rangle>0  }}
  \frac{\sin {\kappa_\alpha} (\langle \hat{\rho}_{\text{g}}+\mu,\alpha^\vee \rangle+\text{g}_\alpha)}
          {\sin {\kappa_\alpha}    (\langle \hat{\rho}_{\text{g}}+\mu,\alpha^\vee \rangle )     }
\prod_{\substack{\alpha\in \hat{R}^+\\ \langle \nu, \alpha^\vee\rangle=2  }}
  \frac{\sin {\kappa_\alpha} (\langle \hat{\rho}_{\text{g}}+\mu,\alpha^\vee \rangle+1+\text{g}_\alpha)}
          {\sin {\kappa_\alpha}    (\langle \hat{\rho}_{\text{g}}+\mu,\alpha^\vee \rangle +1)     }\\
&
\times\prod_{\substack{\alpha\in \hat{R}^+\\ \langle \nu,\alpha^\vee \rangle<0  }}
  \frac{\sin {\kappa_\alpha} (\langle \hat{\rho}_{\text{g}}+\mu,\alpha^\vee \rangle-\text{g}_\alpha)}
          {\sin {\kappa_\alpha}    (\langle \hat{\rho}_{\text{g}}+\mu,\alpha^\vee \rangle )     }
\prod_{\substack{\alpha\in \hat{R}^+\\ \langle \nu, \alpha^\vee \rangle=-2  }}
  \frac{\sin {\kappa_\alpha} (\langle \hat{\rho}_{\text{g}}+\mu,\alpha^\vee \rangle-1-\text{g}_\alpha)}
          {\sin {\kappa_\alpha}    (\langle \hat{\rho}_{\text{g}}+\mu,\alpha^\vee \rangle-1)     }\\
&=\hat{\Delta}(\mu)V_\nu(\hat{\rho}_{\text{g}}+\mu) .
\end{align*}
\end{proof}

\begin{proposition}[Adjoint]\label{adjoint:prp}
For any $\omega\in\hat{P}^+$ small, the finite Macdonald difference operator $D_\omega$ \eqref{D-finite} satisfies the
adjointness relation
\begin{equation}\label{adjoint:rel}
\langle D_\omega f,g\rangle_{\hat{\Delta}}=\langle  f, D_{\omega^*}g\rangle_{\hat{\Delta}}\qquad (f,g\in \ell^2(\hat{\rho}_{\text{g}}+\hat{P}_c,\hat{\Delta})).
\end{equation}
Here $\omega^*:=-w_0\omega$ ($\in\hat{P}^+$), where $w_0$ refers to the {\em longest element} of $W$.
\end{proposition}
\begin{proof}
The stated equality is readily inferred via the following sequence
of elementary manipulations:
\begin{align*}
&\langle f,  D_{\omega^*}g\rangle_{\hat{\Delta}} =\sum_{\substack{\mu\in \hat{P}_{ c}}}
f(\hat{\rho}_{\text{g}}+\mu) \overline{(D_{\omega^*} g) (\hat{\rho}_{\text{g}}+\mu)}\hat{\Delta}(\mu)\\
&=\sideset{}{'}\sum_{\substack{\nu\in \hat{P}(\omega^*)\\ \eta\in W_\nu(w_\nu^{-1}\omega^*)   }}
\sum_{\substack{\mu\in \hat{P}_{ c}\\ \mu+\nu\in \hat{P}_{ c}}}
f(\hat{\rho}_{\text{g}}+\mu)  V_{\nu}(\hat{\rho}_{\text{g}}+\mu) U_{\nu,\eta}(\hat{\rho}_{\text{g}}+\mu) \overline{g(\hat{\rho}_{\text{g}}+\mu+\nu)}\hat{\Delta}(\mu)\\
&\stackrel{(i)}{=}\!\!\!\sideset{}{'}\sum_{\substack{\nu\in \hat{P}(\omega)\\ \eta\in W_\nu(w_\nu^{-1}\omega)   }}
\sum_{\substack{\mu\in \hat{P}_{ c}\\ \mu-\nu\in \hat{P}_{ c}}}
f(\hat{\rho}_{\text{g}}+\mu)  V_{-\nu}(\hat{\rho}_{\text{g}}+\mu) U_{\nu,-\eta}(\hat{\rho}_{\text{g}}+\mu) \overline{g(\hat{\rho}_{\text{g}}+\mu-\nu)}\hat{\Delta}(\mu)\\
&\stackrel{(ii)}{=}\!\!\!\!\!\!\sideset{}{'}\sum_{\substack{\nu\in \hat{P}(\omega)\\ \eta\in W_\nu(w_\nu^{-1}\omega)   }}
\sum_{\substack{\mu\in \hat{P}_{ c}\\ \mu+\nu\in \hat{P}_{ c}}}
f(\hat{\rho}_{\text{g}}+\mu+\nu)  V_{-\nu}(\hat{\rho}_{\text{g}}+\mu+\nu) U_{\nu,-\eta}(\hat{\rho}_{\text{g}}+\mu) \overline{g(\hat{\rho}_{\text{g}}+\mu)}\hat{\Delta}(\mu+\nu)\\
&\stackrel{(iii)}{=}\!\!\!\!\!\!\sideset{}{'}\sum_{\substack{\nu\in \hat{P}(\omega)\\ \eta\in W_\nu(w_\nu^{-1}\omega)   }}
\sum_{\substack{\mu\in \hat{P}_{ c}\\ \mu+\nu\in \hat{P}_{ c}}}
f(\hat{\rho}_{\text{g}}+\mu+\nu)  V_{\nu}(\hat{\rho}_{\text{g}}+\mu) U_{\nu,\eta}(\hat{\rho}_{\text{g}}+\mu) \overline{g(\hat{\rho}_{\text{g}}+\mu)}\hat{\Delta}(\mu)\\
&=
\sum_{\substack{\mu\in \hat{P}_{ c}}}
(D_\omega f) (\hat{\rho}_{\text{g}}+\mu) \overline{g(\hat{\rho}_{\text{g}}+\mu)}\hat{\Delta}(\mu)=\langle  D_\omega f,g\rangle_{\hat{\Delta}} .
\end{align*}
Here we have used:
$(i)$ that $\hat{P}(\omega^*)=-\hat{P}(\omega)$, $W_{-\nu}(w_{-\nu}^{-1}\omega^*)=-W_\nu(w_\nu^{-1}\omega)$ (as
$w_{-\nu}=w_0 w_\nu$ and $W_{-\nu}=W_\nu$) and  $U_{-\nu,-\eta}=U_{\nu,-\eta}$,
 $(ii)$ a translation of $\mu$ by $\nu$ and the equality
 $U_{\nu,-\eta}(\hat{\rho}_{\text{g}}+\mu+\nu) =U_{\nu,-\eta}(\hat{\rho}_{\text{g}}+\mu)$,
 $(iii)$ Lemma \ref{recurrence:lem} and the identity
 \begin{equation*}
\sum_{\eta\in W_\nu(w_\nu^{-1}\omega)  }   U_{\nu,-\eta}=\sum_{\eta\in W_\nu(w_\nu^{-1}\omega)  }   U_{\nu,w_0\eta}=
\sum_{\eta\in W_\nu(w_\nu^{-1}\omega)  }   U_{\nu,\eta} .
\end{equation*}
\end{proof}

\begin{remark}\label{positivity:rem}
The recurrence in Lemma \ref{recurrence:lem} determines the value of $\hat{\Delta}(\mu+\nu)$ in terms of $\hat{\Delta}(\mu)$ (and vice versa), as the coefficients
$V_{-\nu}(\hat{\rho}_{\text{g}}+\mu+\nu)$ and $V_\nu(\hat{\rho}_{\text{g}}+\mu)$ on both sides do not vanish (cf. Lemma \ref{bc:lem}).
A careful examination
of the proofs of Lemmas \ref{regular:lem} and \ref{bc:lem} confirms that these coefficients are in fact always positive.
\end{remark}

\subsection{Orthogonality}
We now interpret the polynomials in the finite-dimensional subspace of $\mathbb{C}[P]^W$ spanned by $m_\lambda$, $\lambda\in P_c$ as elements of
$\ell^2(\hat{\rho}_{\text{g}}+\hat{P}_c,\hat{\Delta})$ by restricting the polynomial variable $x$ to the finite lattice $\hat{\rho}_{\text{g}}+\hat{P}_c\subset E$.
Our main concern is to show that the Macdonald polynomials $p_\lambda$, $\lambda\in P_c$ then constitute an orthogonal
eigenbasis of $D_\omega$ \eqref{D-finite} in this Hilbert space. Recall in this connection that for $\lambda\in P_c$ the inequality  $\mu<\lambda$ implies that $\mu\in P_c$ (cf. the proof of Proposition \ref{meromorphy:prp}), i.e. the polynomials $m_\lambda$, $\lambda\in P_c$  and
$p_\lambda$, $\lambda\in P_c$ span the same subspace of $\mathbb{C}[P]^W$.
At this point it has only been demonstrated that the Macdonald polynomials are meromorphic in $\text{g}>0$ (cf. Proposition \ref{meromorphy:prp}). For the moment all proofs in this subsection therefore assume that the positive root multiplicity parameters are {\em generic}  (and thus in particular regular) avoiding possible poles of the expansion coefficients and degeneracies of the eigenvalues (cf. Lemma \ref{nondegeneracy:lem}), while the statements of the propositions are formulated more generally for all $\text{g}>0$.
A simple continuity argument in the next subsection will remove this discrepancy.

\begin{proposition}[Diagonalization]\label{diagonal:prp}
Let $\omega\in\hat{P}^+$ be small. For any $\text{g}>0$, the Macdonald polynomials $p_\lambda$, $\lambda\in P_c$ form a basis of eigenfunctions
for $D_\omega$ \eqref{D-finite} in $\ell^2(\hat{\rho}_{\text{g}}+\hat{P}_c,\hat{\Delta})$:
\begin{equation}\label{ev-discrete}
D_\omega p_\lambda=E_\omega (\rho_{\text{g}}+\lambda)p_\lambda,\qquad \lambda\in P_c,
\end{equation}
where the eigenvalues are given by $E_\omega$ \eqref{spectrum}.
\end{proposition}
\begin{proof}
By virtue of Proposition \ref{meromorphy:prp}---the Macdonald polynomial $p_\lambda$ satisfies
the eigenvalue equation in Eqs. \eqref{EV:omega}, \eqref{spectrum} as a meromorphic identity in the positive root multiplicity parameter(s). Let us pick $\text{g}>0$ generic (see above). The eigenvalue equation in question then reduces to Eq. \eqref{ev-discrete} upon restriction of
the polynomial variable $x$ to $\hat{\rho}_{\text{g}}+\hat{P}_c$ (by the argument following Eq. \eqref{regular-g}).
From the principal specialization formula in Eq. \eqref{specialization:trigonometric} and the estimates in Lemma \ref{mb:lem} it is moreover seen that $p_\lambda (\hat{\rho}_{\text{g}})>0$ for $\lambda\in P_c$ (as the arguments of the sine functions in the product formula again take values between $0$ and $\pi$), so  $p_\lambda$ constitutes for such $\lambda$ a true (i.e. nonvanishing) eigenfunction of $D_\omega$ in $\ell^2(\hat{\rho}_{\text{g}}+\hat{P}_c,\hat{\Delta})$. The nondegeneracy of the eigenvalues in Lemma \ref{nondegeneracy:lem} furthermore implies that the eigenfunctions $p_\lambda$, $\lambda\in P_c$ are linearly independent in $\ell^2(\hat{\rho}_{\text{g}}+\hat{P}_c,\hat{\Delta})$. Hence, they form a basis of this Hilbert space as
$\dim \ell^2(\hat{\rho}_{\text{g}}+\hat{P}_c,\hat{\Delta})=|\hat{P}_c|=|P_c|$ (cf. Remark \ref{dimension:rem} below).
\end{proof}

\begin{proposition}[Orthogonality]\label{orthogonal:prp}
For any $\text{g}>0$, the Macdonald polynomials $p_\lambda$, $\lambda\in P_c$ form an orthogonal
basis of $\ell^2(\hat{\rho}_{\text{g}}+\hat{P}_c,\hat{\Delta})$:
\begin{equation}
\langle p_\lambda, p_{\tilde{\lambda}}\rangle_{\hat{\Delta}}=
0\qquad \text{iff}\ \lambda\neq \tilde{\lambda}
\end{equation}
($\lambda,\tilde{\lambda}\in P_c$).
\end{proposition}
\begin{proof}
Let us assume that $\text{g}>0$ is generic (see above). The adjointness relation in Proposition \ref{adjoint:prp} and
the eigenvalue equation in Proposition \ref{diagonal:prp} then lead to the stated orthogonality via a standard argument involving the nondegeneracy of the eigenvalues in Lemma \ref{nondegeneracy:lem}:
$$
0=\langle D_\omega p_\lambda,p_{\tilde{\lambda}}\rangle_{\hat{\Delta}}-\langle p_\lambda, D_{\omega^*} p_{\tilde{\lambda}}\rangle_{\hat{\Delta}}=
(E_\omega (\rho_{\text{g}}+\lambda)-E_\omega (\rho_{\text{g}}+\tilde{\lambda}))\langle p_\lambda,p_{\tilde{\lambda}}\rangle_{\hat{\Delta}}
$$
(using that $E_{\omega^*}=\overline{E}_{\omega}$), i.e. $\langle p_\lambda,p_{\tilde{\lambda}}\rangle_{\hat{\Delta}}=0$ if $\lambda\neq \tilde{\lambda}$ because in this situation $E_\omega (\rho_{\text{g}}+\lambda)\neq E_\omega (\rho_{\text{g}}+\tilde{\lambda})$ for some $\omega\in\hat{P}^+$ small. (Notice also that $\langle p_\lambda, p_\lambda\rangle_{\hat{\Delta}}\geq |p_\lambda (\hat{\rho}_{\text{g}})|^2>0$ by the principal specialization formula \eqref{specialization:trigonometric}.)
\end{proof}

\begin{remark}\label{dimension:rem}
Let
$\varphi^\vee= k_1\alpha_1^\vee+\cdots +k_n\alpha_n^\vee$ and
$\vartheta^\vee=m_1\alpha_1^\vee+\cdots +m_n\alpha_n^\vee$ be the decompositions of $\varphi^\vee$ and $\vartheta^\vee$ with respect
to the simple coroots of $R$. Then the generating function for the cardinalities of $P_c$, $c=0,1,2,\ldots$ reads
\begin{equation*}
\sum_{c=0}^\infty |P_c|\, z^c =(1-z)^{-1}\times
\begin{cases}\prod_{j=1}^n (1-z^{k_j})^{-1}&\text{if}\ \hat{R}=R\\
\prod_{j=1}^n (1-z^{m_j})^{-1}&\text{if}\ \hat{R}=R^\vee
\end{cases}
\end{equation*}
($|z|<1$). In particular, one always has that
$|P_c|=|\hat{P}_c|$.
\end{remark}

\subsection{Analyticity}
The triangularity of the monomial expansion of $p_\lambda$ in Eq. \eqref{mp-d1} and the orthogonality
in Proposition \ref{orthogonal:prp} implies that (for generic $\text{g}>0$):
\begin{equation}
 p_\lambda = m_\lambda  -
 \sum_{\substack{\mu\in P_c\\ \mu< \lambda}}
 \frac{\langle m_\lambda, p_\mu\rangle_{\hat{\Delta}} }{\langle p_\mu, p_\mu\rangle_{\hat{\Delta}}} p_\mu
 \qquad (\lambda\in P_c).
\end{equation}
From this Gram-Schmidt type formula it is manifest---by induction on the dominant weight $\lambda$ with respect to the dominance ordering---that $p_\lambda$ is in fact {\em analytic} in $\text{g}>0$ (since the positive weight function $\hat{\Delta}$ is analytic in $\text{g}>0$ and the denominators $\langle p_\mu, p_\mu\rangle_{\hat{\Delta}}$ remain positively bounded from below). As a  consequence, the statements in
Propositions \ref{diagonal:prp} and \ref{orthogonal:prp} extend from generic $\text{g}$ to the full parameter domain $\text{g}>0$ by continuity.

\begin{remark}
It is an immediate consequence of Proposition \ref{orthogonal:prp} that the matrix
$[m_\lambda (\hat{\rho}_{\text{g}}+\mu)]_{\lambda\in P_c,\mu\in\hat{P}_c}$ is invertible, i.e. the evaluation homomorphism mapping
the subspace of $\mathbb{C}[P]^W$ spanned by $m_\lambda$, $\lambda\in P_c$ into $\ell^2(\hat{\rho}_{\text{g}}+\hat{P}_c,\hat{\Delta})$ is a linear isomorphism.
\end{remark}

\section{Normalization}\label{sec5}

\subsection{Finite Pieri identity}
By combining Proposition \ref{diagonal:prp} with the duality symmetry in Eq. \eqref{duality-symmetry}, one arrives at the
following finite Pieri identity in $\ell (\hat{\rho}_g+\hat{P}_c,\hat{\Delta})$
associated with $\omega\in P^+$ small:
 \begin{align} \label{pieri}
 \hat E_\omega P_\lambda
=\sum_{\substack{\nu\in P(\omega)\\  \lambda+\nu\in P_c}} \sideset{}{'}\sum_{\eta\in W_\nu(w_\nu^{-1}\omega)}
 \hat V_\nu(\rho_{\text{g}}+\lambda )\hat U_{\nu,\eta}(\rho_{\text{g}}+\lambda )P_{\lambda+\nu} .
\end{align}
Indeed---upon evaluating the difference equation $\hat{D}_\omega \hat{P}_\mu = \hat{E}_\omega (\hat{\rho}_g+\mu)  \hat{P}_\mu $ for $\mu\in\hat{P}_c$ at  $\rho_{\text{g}}+\lambda$, $\lambda\in P_c$ and invoking of the duality
symmetry---Eq. \eqref{pieri} follows immediately.
With the aid of the above Pieri identity and the recurrence in Lemma \ref{recurrence:lem}, it is not difficult to express the quadratic norms $\langle P_\lambda , P_\lambda\rangle_{\hat{\Delta}} $ in terms of the norms of the unit polynomial
$\langle1 , 1\rangle_{\hat{\Delta}} $.
For this purposes it suffices to restrict attention to the Pieri identities associated with the (quasi-)minuscule weights only.

\begin{lemma}[(Quasi-)Minuscule Path Connectedness]\label{path:lem}
For any $\lambda\in P_c$, there exists a path $0=\lambda^{(0)}\to\lambda^{(1)}\to\cdots \to\lambda^{(\ell )}=\lambda$ of weights in $P_c$
such that the increments $\lambda^{(k)}-
\lambda^{(k-1)}$, $k=1,\ldots ,\ell$ are given either by positive roots in the orbit $W\vartheta$ or by
minuscule weights.
\end{lemma}
\begin{proof}
From the tables in Bourbaki \cite{bou:groupes}, it is readily inferred that the fundamental weights $\omega_1,\ldots,\omega_n$ of $R$ can be grouped---by means of the dominance order on $P^+$---in $\text{Ind}(R)$ linearly ordered chains with minimal elements given by the (quasi-)minuscule fundamental weights. Subsequent fundamental weights in a chain  differ moreover by a root in $W\vartheta$.
The existence of the path claimed by the lemma is clear if the decomposition of $\lambda\in P_c$ in the basis of the fundamental weights
$$
\lambda=\lambda_1\omega_1+\cdots +\lambda_n\omega_n
$$
contains at most nonzero coefficients corresponding to fundamental weights that are either minuscule or quasi-minuscule. Otherwise, if $\lambda_j>0$ with $\omega_j$ neither minuscule nor quasi-minuscule, then the weight $\tilde{\lambda}$ obtained by subtracting the positive root $\omega_j-\omega_{j^\prime}\in W\vartheta$---where $\omega_{j^\prime}$ refers to the fundamental weight preceding $\omega_j$ in the respective chain---belongs to $ P_c$
 as
 $\tilde{\lambda}<\lambda$ (cf. the proof of Proposition \ref{meromorphy:prp}). The lemma now follows by induction with respect to the dominance order on $P^+$.
\end{proof}

After these preparations the computation of the quadratic norms is standard.

\begin{proposition}[Normalization] For any $\lambda\in P_c$, the quadratic norm of the normalized Macdonald polynomial $P_\lambda$ \eqref{P-reno}, \eqref{specialization:trigonometric} is given by
\begin{equation}
\langle P_\lambda , P_\lambda\rangle_{\hat{\Delta}} = \frac{\mathcal{N}_0}{\Delta (\lambda) }\quad \text{with}
\quad \mathcal{N}_0=\langle 1 ,1\rangle_{\hat{\Delta}} .
\end{equation}
\end{proposition}

\begin{proof}
 For $\omega\in P^+$  (quasi)-minuscule, $\lambda\in P_c$ and $\nu\in W\omega$  such that $\lambda+\nu\in P_c$, an expansion of the products on both sides of the identity
 \begin{equation*}
   \langle \hat E_\omega P_\lambda, P_{\lambda+\nu} \rangle_{\hat{\Delta}} =    \langle P_\lambda, \hat E_{\omega^*} P_{\lambda+\nu} \rangle_{\hat{\Delta}}
 \end{equation*}
by means of the corresponding RHS of the Pieri formula \eqref{pieri} entails (using the orthogonality of
Proposition \ref{orthogonal:prp}) that
\begin{equation*}
  \hat V_\nu(\lambda+\rho_{\text{g}})    \langle P_{\lambda+\nu}, P_{\lambda+\nu} \rangle_{\hat{\Delta}}
  =      \hat V_{-\nu}(\lambda+\nu+\rho_{\text{g}})   \langle P_\lambda, P_\lambda \rangle_{\hat{\Delta}}
\end{equation*}
(because  $\eta=\nu$ if $\nu\in W\omega$, and $\hat{U}_{\nu,\nu}=1$).
This relation can be recasted---with the aid of the dual version of the recurrence in Lemma \ref{recurrence:lem}---in terms of the following translational symmetry:
\begin{equation*}
 \Delta(\lambda+\nu)\langle P_{\lambda+\nu}, P_{\lambda+\nu} \rangle_{\hat{\Delta}}
=
\Delta(\lambda) \langle P_\lambda, P_\lambda \rangle_{\hat{\Delta}}  .
\end{equation*}
By applying the translational symmetry along the increments of a path in
Lemma \ref{path:lem}, we conclude that $\Delta(\lambda)\langle P_{\lambda}, P_{\lambda} \rangle_{\hat{\Delta}}$ is equal to  $\langle 1, 1 \rangle_{\hat{\Delta}}$ (and thus independent of $\lambda$), i.e.
$
\langle P_{\lambda}, P_{\lambda} \rangle_{\hat{\Delta}}
=\langle 1,1\rangle_{\hat{\Delta}}  /\Delta(\lambda)$.
\end{proof}

\subsection{Total mass of the weight function} In this subsection we  momentarily allow the multiplicity parameter $\text{g}$ to be negative and even complex valued.
When $\hat{R}=R^\vee$ the product formula for the total mass of the weight function  in Section \ref{sec3} follows from the trigonometric identity in \cite[\text{Rem.}~4.6]{maz:finite}. This trigonometric identity was obtained by truncating a basic hypergeometric summation formula due to Aomoto, Ito and Macdonald
\cite{aom:elliptic,ito:symmetry,mac:formal}. It is straightforward to adapt the techniques of Ref. \cite{maz:finite} to incorporate the case that $\hat{R}=R$. Indeed, the appropriate Aomoto-Ito-Macdonald sum for our purposes reads (cf. \cite[\text{Sec.}~1]{aom:elliptic}, \cite[\text{Sec.}~4]{ito:symmetry}, and
\cite[\text{Sec.}~9]{mac:formal}):
\begin{subequations}
\begin{multline}\label{AIM-S}
 \sum_{\lambda\in P}
q^{-2\langle   \hat \rho_{\text{g}}, \lambda\rangle }
 \prod_{\alpha\in R^+}
\Bigl(
\frac  {  1-q_\alpha ^{\langle   x+\lambda ,\alpha^\vee\rangle}}
          {  1- q_\alpha^{\langle    x,\alpha^\vee \rangle}  }
\Bigr)
 \frac{  (q_\alpha^{\langle   x , \alpha^\vee\rangle + {\text{g}}_\alpha } ; q_\alpha)_{\langle   \lambda, \alpha^\vee\rangle }}
         { (q_\alpha^{\langle    x ,\alpha^\vee\rangle+1 -{\text{g}}_\alpha} ; q_\alpha)_{\langle  \lambda, \alpha^\vee\rangle }  }
\\
=
\mathcal N
\prod_{\alpha\in R}
\frac  {  (q_\alpha^{\langle  x,\alpha^\vee\rangle + 1} ; q_\alpha)_\infty }
         { (q_\alpha^{\langle   x,\alpha^\vee\rangle + 1 - {\text{g}}_\alpha} ; q_\alpha)_\infty    } ,
\end{multline}
with $x\in E$, $0<q<1$ and
\begin{equation}\label{AIM-P}
\mathcal N:= \text{Ind}(R)\prod_{\alpha\in \hat{R}^+}
  \frac{(q_\alpha^{-\langle  \hat\rho_{\text{g}}, \alpha^\vee\rangle +1 - {\text{g}}_\alpha},
                q_\alpha^{-\langle   \hat\rho_{\text{g}}, \alpha^\vee\rangle + {\text{g}}_\alpha + \delta_\alpha} ; q_\alpha)_\infty }
     {( q_\alpha^{-\langle   \hat\rho_{\text{g}}, \alpha^\vee\rangle +1}, q_\alpha^{-\langle  \hat \rho_{\text{g}}, \alpha^\vee\rangle} ; q_\alpha)_\infty } ,
\end{equation}
\end{subequations}
where the value of $\delta_\alpha$ is $1$ if $\alpha$ is simple and $0$ otherwise. Here we have employed the convention that  $(a;q)_m:=(a;q)_\infty/(aq^m;q)_\infty$ for $m<0$.
The basic hypergeometric sum in Eqs. \eqref{AIM-S}, \eqref{AIM-P}
is normalized such that the term on the LHS is equal to $1$ when $\lambda=0$. To ensure convergence and avoid poles
it is assumed that
$\text{g}<0$ and that for all $\alpha\in R$: $\langle x, \alpha^\vee\rangle\neq 0$ and ${\text{g}}_\alpha-\langle x, \alpha^\vee\rangle  \not\in \mathbb{N}$.

At $x=\rho_{\text{g}}$ with $\text{g}<0$ such that
${\text{g}}_\alpha-\langle \rho_{\text{g}}, \alpha^\vee\rangle  \not\in \mathbb{N}$ for all $\alpha\in R^+$,
this Aomoto-Ito-Macdonald sum reduces to a sum over the dominant weights:
\begin{multline}\label{AIM+}
 \sum_{\lambda\in   P^+}
q^{-2\langle   \hat \rho_{\text{g}} , \lambda\rangle }
 \prod_{\alpha\in R^+}
\Bigl(
\frac  {  1-q_\alpha ^{\langle  \rho_{\text{g}}+\lambda , \alpha^\vee \rangle}}
          {  1- q_\alpha^{\langle    \rho_{\text{g}} ,  \alpha^\vee\rangle}  }
\Bigr)
 \frac{  (q_\alpha^{\langle   \rho_{\text{g}},  \alpha^\vee  \rangle + {\text{g}}_\alpha } ; q_\alpha)_{\langle   \lambda, \alpha^\vee\rangle }}
         { (q_\alpha^{\langle \rho_{\text{g}}, \alpha^\vee \rangle+1 - {\text{g}}_\alpha} ; q_\alpha)_{\langle  \lambda,\alpha^\vee\rangle }  }  \\
=\text{Ind}(R) \prod_{\alpha\in R^+}
  \frac{( q_\alpha^{\langle  \rho_{\text{g}},  \alpha^\vee\rangle +1}; q_\alpha)_\infty}
          {  (q_\alpha^{-\langle\hat \rho_{\text{g}},   \hat \alpha^\vee \rangle}; q_\alpha)_\infty }
 \prod_{\alpha\in R^+\backslash I}
  \frac{(q_\alpha^{-\langle \hat\rho_{\text{g}},  \hat \alpha^\vee \rangle  + {\text{g}}_\alpha}; q_\alpha)_\infty}
          {(q_\alpha^{\langle \rho_{\text{g}},  \alpha^\vee\rangle+1 - {\text{g}}_\alpha}; q_\alpha)_\infty}  ,
 \end{multline}
where $I\subseteq R^+$ refers to the basis of the simple roots. Indeed,
for $\lambda\not\in  P^+$ there exists a simple root $\beta$ such that $\langle \lambda,\beta^\vee\rangle<0$. The term on the LHS then picks up a zero
 from the factor
$$
\frac{1}{(q_{\beta}^{\langle   \rho_{\text{g}}, \beta^\vee \rangle+1 - g_\beta} ; q_\beta)_{\langle  \lambda,\beta^\vee\rangle } }
=\frac{1}{(q_{\beta};q_{\beta})_{\langle  \lambda,\beta^\vee\rangle } }=0 .
$$

By exploiting the analyticity in $\text{g}$, the summation formula in Eq. \eqref{AIM+} can be extended to complex $\text{g}$ with $\text{Re}(\text{g})<0$ such that ${\text{g}}_\alpha-\langle \rho_{\text{g}}, \alpha^\vee\rangle  \not\in (\mathbb{N}+2\pi i \mathbb{Z}/\log q_\alpha )$ for all $\alpha\in R^+$. Upon choosing such $\text{g}$ subject to the additional constraint  that
$h_{\text{g}}+c=2\pi i/\log q_\varphi$, so $q_\varphi^{h_{\text{g}}+c}=1$, the LHS of Eq. \eqref{AIM+} truncates to a finite sum of the form:
$$
 \sum_{\lambda\in   P_c}
q^{-2\langle  \hat \rho_{\text{g}} ,  \lambda \rangle }
 \prod_{\alpha\in R^+}
\Bigl(
\frac  {  1-q_\alpha ^{\langle  \rho_{\text{g}}+\lambda ,  \alpha^\vee\rangle}}
          {  1- q_\alpha^{\langle   \rho_{\text{g}} ,  \alpha^\vee\rangle}  }
\Bigr)
 \frac{  (q_\alpha^{\langle   \rho_{\text{g}} ,  \alpha^\vee \rangle + {\text{g}}_\alpha } ; q_\alpha)_{\langle   \lambda, \alpha^\vee\rangle }}
         { (q_\alpha^{\langle  \rho_{\text{g}} ,  \alpha^\vee\rangle+1 - {\text{g}}_\alpha} ; q_\alpha)_{\langle  \lambda,\alpha^\vee\rangle }  } ,
$$
since  for $\lambda\in P^+\setminus P_c$ one has that
$\langle \lambda, \hat \psi^\vee\rangle >  c$, i.e. the corresponding term on the LHS of Eq. \eqref{AIM+} picks up a zero from the factor
$$
( q_{\hat\psi}^{\langle   \rho_{\text{g}} , \hat \psi^\vee\rangle +g_\psi} ; q_{\hat\psi})_{\langle   \lambda, \hat \psi^\vee\rangle }
=(q_\varphi^{h_{\text{g}}};q_\varphi)_{\langle   \lambda, \hat \psi^\vee\rangle}=
( q_\varphi^{-c} ; q_\varphi)_{\langle   \lambda, \hat \psi^\vee\rangle}
=0 .
$$
In this situation the RHS of Eq. \eqref{AIM+} can be reduced accordingly  to a quotient of finite $q$-factorials, upon canceling common factors in the numerator and the denominator with the aid of the relation $q_\varphi^{h_{\text{g}}+c}=1$. By passing from $q$-factorials to trigonometric factorials via the substitution $q=e^{2i\kappa}$ with $\kappa= \pi/(u_\varphi (h_{\text{g}}+c))$, one ends up with the summation formula
\begin{equation}
\sum_{\lambda\in P_c}\Delta (\lambda)=\text{Ind}(R)\mathcal{N}_c,
\end{equation}
with $\mathcal{N}_c$ given by the tables in Section \ref{sec3}.  The product formula for the total mass of the weight function then follows by continuing analytically to the parameter domain $\text{g}>0$.

\begin{remark}\label{E7-6c:rem}
Throughout the paper it was assumed that $c$ is not a proper multiple of $6$ when $R=E_7$. The reason being that in this particular situation simultaneous degenerations in the spectrum of $D_\omega$ \eqref{D-finite} with $\omega\in\hat{P}^+$ small do occur (cf. the appendix below), causing Lemma \ref{nondegeneracy:lem} (and thus the proof of Theorem \ref{main:thm}) to break down at this point (only). In fact, our proof of the orthogonality relations in Eq. \eqref{ort-rel2} applies verbatim in this situation for weights belonging to
$P_{\tilde{c}}\subseteq P_c$ with $\tilde{c}=\lceil\frac{11}{12}c\rceil$, in view of Remark \ref{E7-degenerations}. To extend the proof under consideration to the complete basis of Macdonald polynomials for $\ell^2(\hat{P}_c,\hat{\Delta})$, one more independent commuting difference operator is needed to separate the spectrum. In principle the complete algebra of commuting difference operators containing the explicit Macdonald difference operators $\mathcal{D}_\omega$ \eqref{D}--\eqref{U} can be obtained from Cherednik's representation of the double affine Hecke algebra \cite[\text{Eqs.}~(4.4.12), (5.3.3)]{mac:affine}. In this approach it therefore suffices to verify that these difference operators restrict to normal operators in the  finite-dimensional Hilbert space $\ell^2(\hat{P}_c,\hat{\Delta})$, cf. \cite[\text{Secs.}~4.5, 5.3]{mac:affine}.

\end{remark}

\appendix

\section{Nondegeneracy of the Eigenvalues for exceptional root systems}\label{appA}
In this appendix the nondegeneracy of the eigenvalues in Lemma \ref{nondegeneracy:lem} is verified for the exceptional root systems. Specifically, we will check that if for certain $\lambda ,\mu \in P_c$ the equality
\begin{equation}\label{Ev-id}
\hat{m}_\omega (\rho_{\text{g}}+\lambda )= \hat{m}_\omega (\rho_{\text{g}}+\mu)
\end{equation}
holds for all $\omega\in \hat{P}^+$ small (as an identity in $\text{g}$), then necessarily $\lambda=\mu$.
This implies that the same holds true for the equality $E_\omega (\rho_{\text{g}}+\lambda )= E_\omega (\rho_{\text{g}}+\mu)$
in view of the triangularity of $E_\omega$ \eqref{spectrum} with respect to the monomial basis.

Since for $R$ exceptional the dual root system $R^\vee$ is isomorphic to $R$, the truncation relation in Remark~\ref{tr:rem} reads $t_\vartheta^{h/2}t_\varphi^{h/2}q_\varphi^c=1$ (with
$h$ being the Coxeter number of $R$). We write $\tilde{h}$ for the Coxeter number of
the simply laced subsystem $W\varphi\subseteq R$ (so $\tilde{h}=h$ if $R$ is simply laced and $\tilde{h}=h/2$---in our situation---if $R$ is multiply laced). Let us furthermore denote the primitive root of unity $e^{2\pi i/\tilde{h}}$ by $\varepsilon$.
Upon writing $\hat{m}_\omega(\rho_{\text{g}}+\lambda)=\sum_{\nu\in W\omega} q^{\langle \nu ,\lambda\rangle}\prod_{\alpha\in R^+}t_\alpha^{\langle\nu,\hat{\alpha}^\vee\rangle/2}$ and
elimination of $t_\vartheta$  by means of the relations $t_\vartheta =\varepsilon q_\varphi^{-c/\tilde{h}}$ if $R$ is simply laced or
$t_\vartheta  t_\varphi=\varepsilon q_\varphi^{-c/\tilde{h}}$
if $R$ is multiply laced, both sides of the equality in Eq. \eqref{Ev-id} become Laurent polynomials in $t_\varphi$
with coefficients built of terms that are products of powers of $\varepsilon$ and $q$ (so the Laurent polynomials in question are of degree zero if $R$ is simply laced).
For $R$ multiply laced both sides of Eq. \eqref{Ev-id} are equal as analytic functions in $\text{g}$ iff all coefficients of the corresponding Laurent polynomials in $t_\varphi$ match. (Indeed,
the polar angles of
$q=\exp (\frac{2\pi i}{u_\varphi (h_{\text{g}}+c)})$ and $t_\varphi=q^{u_\varphi \text{g}_\varphi}=\exp (\frac{2\pi i \text{g}_\varphi }{ h_{\text{g}}+c})$ are controlled  by two independent parameters $\text{g}_\vartheta$ and $\text{g}_\varphi$, so by varying these parameters over the positive reals the tuple of the respective angles
covers an open subset of $(0,\frac{2\pi }{u_\varphi c})\times (0,\frac{2\pi}{\tilde{h}})$.)

The expressions (for the coefficients of the Laurent polynomials in $t_\varphi$) on both sides of Eq. \eqref{Ev-id} are themselves polynomials in the primitive root of unity $\varepsilon$ of degree $\leq \tilde{h}-1$ (possibly up to an overall factor $\varepsilon^{1/2}$ when $\text{Ind}(R)> 1$), with coefficients that are sums of powers of $q$.
To eliminate linear dependencies between these roots of unity, the powers
$\varepsilon^{\phi(\tilde{h})},\ldots ,\varepsilon^{\tilde{h}-1}$---where $\phi$ refers to Euler's totient function counting the number of coprimes not exceeding its argument---are expressed in terms of the basis
$1,\varepsilon,\ldots,\varepsilon^{\phi(\tilde{h})-1}$ via their residues modulo
the cyclotomic polynomial $\Phi_{\tilde{h}}(\varepsilon)$ of degree $\phi (\tilde{h})$.
Upon differentiating the coefficients with respect to $q$ and subsequently evaluating at $q=1$,
a pairwise comparison of terms from both sides provides linear relations of the form $\langle\lambda-\mu,v\rangle =0$ with $v\in Q^\vee$ (where we exploit the fact that the roots of unity $1,\varepsilon,\ldots,\varepsilon^{\phi(\tilde{h})-1}$ are linearly independent over the rationals).
By varying over the different coefficients and small weights $\omega\in\hat{P}$, we
deduce this way that the equality in Eq. \eqref{Ev-id} implies that $\lambda-\mu$
must be orthogonal to $n$ ($=\text{rank}(R)$)
linearly independent vectors $v\in Q^\vee$ unless $R$ is of type $E_7$, whence  $\mu$ must be equal to $\lambda$ in these cases.

When $R$ is of type $E_7$, the relevant vectors $v\in Q^\vee$ turn out to span a hyperplane, viz.
the equality in Eq. \eqref{Ev-id} now permits to conclude only that $\lambda-\mu$
must belong to the line perpendicular to this hyperplane.
A comparison of the quadratic terms---obtained
by first applying the differential operator $(q\frac{ \text{d}}{\text{d}q})^2$ to the coefficients of the expression on both sides of Eq. \eqref{Ev-id} and then evaluating at $q=1$---under the additional assumption that $\mu$ differs from $\lambda$ by a {\em nonzero} vector belonging to this perpendicular line, now entails a nonhomogeneous linear system for $\lambda$. When $c$ is not a multiple of $6$, its (two-dimensional) solution space does not intersect $P$, whence
the equality in Eq. \eqref{Ev-id} still implies that $\lambda=\mu$ in this situation.

Below we identify for each exceptional root system (ordered by increasing rank), a minimal choice of small weights $\omega$ and the  corresponding coefficients of $\hat{m}_\omega (\rho_{\text{g}}+\lambda)$ giving rise to a maximal system of linearly independent vectors $v\in Q^\vee$ that are orthogonal to $\lambda-\mu$ when Eq. \eqref{Ev-id} holds.
Here the weights $\lambda$ (and $\mu$) will be expressed in the basis of fundamental weights
$\lambda=\lambda_1\omega_1+\cdots +\lambda_n\omega_n$, and the relevant vectors $v\in Q^\vee$ will be represented by the components $(v_1, v_2, \ldots ,v_n)$  with respect to the dual basis of simple coroots (i.e. $v=v_1\alpha_1^\vee+\cdots+v_n\alpha_n^\vee$).
In each case, the normalization of the root system, the choice of the positive subsystem, and the numbering of the elements of the simple and fundamental bases will follow the conventions of the tables in Bourbaki \cite{bou:groupes}.
We end the appendix by providing some details regarding the additional analysis of the quadratic terms required to rule out the degeneracies when $R$ is of type $E_7$.

\subsection{Type $\boldsymbol{G}$} The quasi-minuscule weight $\omega$ of $\hat{R}$ is equal to
$\varphi^\vee$ if $\hat{R}=R^\vee$ and equal to $\vartheta$ is $\hat{R}=R$. For $R$ of type $G_2$, the corresponding monomials  $\hat{m}_{\omega}(\rho_{\text{g}}+\lambda)$ are of the form $\hat{m}_{\omega}(\rho_{\text{g}}+\lambda)=\hat{m}_{\omega}^+(\rho_{\text{g}}+\lambda)+\overline{\hat{m}_{\omega}^+(\rho_{\text{g}}+\lambda)}$ with
\begin{align*}
\hat{m}_{\varphi^\vee}^+(\rho_{\text{g}}+\lambda)&=
t_\vartheta t_\varphi^2q^{\lambda_1+2\lambda_2}+t_\vartheta t_\varphi q^{\lambda_1+\lambda_2}+t_\varphi q^{\lambda_2}
\qquad (\hat{R}=R^\vee ), \\
\hat{m}_{\vartheta}^+(\rho_{\text{g}}+\lambda)&=
t_\vartheta^2 t_\varphi q^{2\lambda_1+3\lambda_2}+t_\vartheta t_\varphi q^{\lambda_1+3\lambda_2}+t_\vartheta q^{\lambda_1}
\qquad (\hat{R}=R ) .
\end{align*}
We have that $\tilde{h}=3$ and $\varepsilon=e^{2\pi i/3}$. Elimination of $t_\vartheta$
via the truncation relation $t_\vartheta t_\varphi =\varepsilon q_\varphi^{-c/3}$
and calculation of the residues modulo the cyclotomic polynomial $\Phi_3(\varepsilon)=\varepsilon^2+\varepsilon+1$ gives
\begin{align*}
\hat{m}_{\varphi^\vee} (\rho_{\text{g}}+\lambda)&=
( q^{\lambda_2} + \varepsilon q^{ \lambda_1 + 2\lambda_2 -\frac{c}{3}}   ) t_\varphi
+(   q^{-\lambda_2} -q^{-\lambda_1- 2 \lambda_2 + \frac{c}{3}}  -  \varepsilon q^{ -\lambda_1 - 2\lambda_2 + \frac{c}{3}} ) t_\varphi^{-1} \\
& + ( - q^{ -\lambda_1 -  \lambda_2 + \frac{c}{3}} + \varepsilon (q^{ \lambda_1 + \lambda_2 -\frac{c}{3}}-q^{ -\lambda_1 - \lambda_2 + \frac{c}{3}}) ) \qquad (\hat{R}=R^\vee )
  \end{align*}
and
\begin{align*}
\hat{m}_{\vartheta} (\rho_{\text{g}}+\lambda)&=
(   -q^{ - \lambda_1 +\frac{c}{3}u_\varphi} +  \varepsilon (q^{  -2 \lambda_1  - 3 \lambda_2 + \frac{2c}{3}u_\varphi }   -  q^{ - \lambda_1 +\frac{c}{3}u_\varphi})  ) t_\varphi \\
&+ (  - q^{ 2 \lambda_1 + 3  \lambda_2 -\frac{2c}{3}u_\varphi}   +  \varepsilon  (q^{ \lambda_1 -\frac{c}{3}u_\varphi } -  q^{ 2 \lambda_1  + 3 \lambda_2  -\frac{2c}{3}u_\varphi}) ) t_\varphi^{-1} \\
&+    (  -q^{-\lambda_1- 3 \lambda_2 +\frac{c}{3}u_\varphi }     +  \varepsilon(q^{ \lambda_1 + 3 \lambda_2  -\frac{c}{3}u_\varphi}  -  q^{ - \lambda_1  - 3 \lambda_2  + \frac{c}{3}u_\varphi})   ) \qquad (\hat{R}=R ) .
\end{align*}
Differentiation with respect to $q$ of
the coefficients of the Laurent polynomials in $t_\varphi$ on both sides of
 Eq. \eqref{Ev-id} and subsequent evaluation at $q=1$ leads---upon
comparing the coefficients of
$t_\varphi$ and $\varepsilon t_\varphi$ from both sides---to
the relations $\lambda_2=\mu_2$, $\lambda_1+2\lambda_2=\mu_1+2\mu_2$ if $\hat{R}=R^\vee$ and
$\lambda_1=\mu_1$, $\lambda_1+3\lambda_2=\mu_1+3\mu_2$ if $\hat{R}=R$. In other words,
the equality in Eq. \eqref{Ev-id} implies that $\lambda-\mu$ must be orthogonal to $\alpha_2^\vee$ and
$\alpha_1^\vee+2\alpha_2^\vee$ if $\hat{R}=R^\vee$ and to
$\alpha_1^\vee$ and
$\alpha_1^\vee+3\alpha_2^\vee$ if $\hat{R}=R$.
In both cases, the equality in  Eq. \eqref{Ev-id}
therefore holds only when
$\lambda=\mu$.

\subsection{Type $\boldsymbol{F}$}
Proceeding as for $G_2$, we compute for $\omega\in\hat{P}^+$ quasi-minuscule $\hat{m}_{\omega}(\rho_{\text{g}}+\lambda)=\hat{m}_{\omega}^+(\rho_{\text{g}}+\lambda)+\overline{\hat{m}_{\omega}^+(\rho_{\text{g}}+\lambda)}$, with $\omega=\varphi^\vee$ and
\begin{eqnarray*}
\lefteqn{\hat{m}_{\varphi^\vee}^+(\rho_{\text{g}}+\lambda)=}
&&\\
&& t_\vartheta ^3 t_\varphi^5 q^{2 \lambda_1 +3 \lambda_2+2 \lambda_3+\lambda_4}
+ t_\vartheta^3 t_\varphi^4 q^{\lambda_1 +3 \lambda_2+2 \lambda_3+\lambda_4}
+ t_\vartheta^3 t_\varphi^3 q^{\lambda_1 +2 \lambda_2+2 \lambda_3+\lambda_4}\\
&&+ t_\vartheta^2 t_\varphi^3 q^{\lambda_1 +2 \lambda_2+\lambda_3+\lambda_4}
+ t_\vartheta^2 t_\varphi^2 q^{\lambda_1 +\lambda_2+\lambda_3+\lambda_4}
+ t_\vartheta t_\varphi^3 q^{\lambda_1 +2 \lambda_2+\lambda_3}
+ t_\vartheta^2 t_\varphi q^{\lambda_2+\lambda_3+\lambda_4}\\
&&+ t_\vartheta t_\varphi^2 q^{\lambda_1 +\lambda_2+\lambda_3}
+ t_\vartheta t_\varphi q^{\lambda_2+\lambda_3}
+ t_\varphi^2 q^{\lambda_1 +\lambda_2}
+ t_\varphi (q^{\lambda_1}+ q^{\lambda_2})
\end{eqnarray*}
if $\hat{R}=R^\vee$, and with $\omega=\vartheta$ and
\begin{eqnarray*}
\lefteqn{\hat{m}_{\vartheta}^+ (\rho_{\text{g}}+\lambda)=}&& \\
&& t_\vartheta^5 t_\varphi^3 q^{2 \lambda_1 +4 \lambda_2+3 \lambda_3+2\lambda_4}
+ t_\vartheta^4 t_\varphi^3 q^{2\lambda_1 +4\lambda_2+3 \lambda_3+\lambda_4}
+ t_\vartheta^3 t_\varphi^3 q^{2\lambda_1 +4 \lambda_2+2 \lambda_3+\lambda_4}\\
&&+  t_\vartheta^3t_\varphi^2 q^{2\lambda_1 +2 \lambda_2+2\lambda_3+\lambda_4}
+ t_\vartheta^3 t_\varphi q^{2\lambda_2+2\lambda_3+\lambda_4}
+ t_\vartheta^2 t_\varphi^2 q^{ 2\lambda_1 +2 \lambda_2+\lambda_3+\lambda_4}\\
&&+ t_\vartheta^2 t_\varphi q^{2\lambda_2+\lambda_3+\lambda_4}
+ t_\vartheta t_\varphi^2 q^{2\lambda_1 +2\lambda_2+\lambda_3}
+ t_\vartheta^2 q^{\lambda_3+\lambda_4}
+ t_\vartheta t_\varphi q^{2\lambda_2 +\lambda_3}
+ t_\vartheta (q^{\lambda_3}
+ q^{\lambda_4})
\end{eqnarray*}
if $\hat{R}=R$. In the present case $\tilde{h}=6$, $\varepsilon=e^{2\pi i/6}$, and elimination of $t_\vartheta$
via $t_\vartheta t_\varphi =\varepsilon q_\varphi^{-c/6}$ yields
modulo the cyclotomic polynomial $\Phi_6(\varepsilon)=\varepsilon^2-\varepsilon+1$:
\begin{multline*}\hat{m}_{\varphi^\vee}(\rho_{\text{g}}+\lambda)=
\Bigl(q^{ \lambda_1 + \lambda_2 }  - q^{2\lambda_1 + 3\lambda_2 + 2 \lambda_3 + \lambda_4 - \frac{c}{2}}  + \varepsilon q^{\lambda_1 + 2\lambda_2 + \lambda_3 -\frac{c}{6}}   \Bigr) t_\varphi^2
\\
+
\Bigl(q^{ \lambda_1} +q^{\lambda_2}  -q^{ \lambda_1+2  \lambda_2+\lambda_3+\lambda_4 -\frac{c}{3}}-q^{\lambda_1+3  \lambda_2+2  \lambda_3+\lambda_4 -\frac{c}{2}}
\\   + \varepsilon  (q^{\lambda_1+2  \lambda_2+\lambda_3+\lambda_4 -\frac{c}{3}}+q^{\lambda_1+\lambda_2+\lambda_3 -\frac{c}{6}}-q^{-\lambda_2-\lambda_3-\lambda_4 +\frac{c}{3}})   \Bigr)    t_\varphi
\\
\Bigl(
-q^{-\lambda_1-2  \lambda_2-2  \lambda_3-\lambda_4 +\frac{c}{2} }+q^{-\lambda_2-\lambda_3 +\frac{c}{6}}-q^{\lambda_1+\lambda_2+\lambda_3+\lambda_4   -\frac{c}{3}  }-q^{\lambda_1+2  \lambda_2+2  \lambda_3+\lambda_4 -\frac{c}{2}}  +  \\
  \varepsilon (q^{\lambda_1+\lambda_2+\lambda_3+\lambda_4 - \frac{c}{3}}+q^{\lambda_2+\lambda_3  -\frac{c}{6} } - q^{-\lambda_2-\lambda_3 +  \frac{c}{6}}-q^{-\lambda_1-\lambda_2-\lambda_3-\lambda_4 +  \frac{c}{3}})
\Bigr)
+
\\
\Bigl(
q^{-\lambda_1-\lambda_2-\lambda_3  + \frac{c}{6}  }-q^{-\lambda_1-3  \lambda_2-2  \lambda_3-\lambda_4  + \frac{c}{2}  } + q^{-\lambda_2} + q^{-\lambda_1} - q^{ \lambda_2+\lambda_3+\lambda_4 -  \frac{c}{3}  }
\\ + \varepsilon (-q^{  -\lambda_1-2  \lambda_2-\lambda_3-\lambda_4  + \frac{c}{3} }+q^{\lambda_2+\lambda_3+\lambda_4   - \frac{c}{3}}-q^{-\lambda_1-\lambda_2-\lambda_3   + \frac{c}{6}})
\Bigr) t_\varphi^{-1}
\\
+
\Bigl(
-q^{-2  \lambda_1-3  \lambda_2-2  \lambda_3-\lambda_4  + \frac{c}{2}}+q^{-\lambda_1-2  \lambda_2-\lambda_3  +  \frac{c}{6}}+q^{-\lambda_1-\lambda_2}  -\varepsilon q^{-\lambda_1-2  \lambda_2-\lambda_3  + \frac{c}{6}}
\Bigr)  t_\varphi^{-2}
 \end{multline*}
if $\hat{R}=R^\vee$, and
\begin{multline*}\hat{m}_{\vartheta} (\rho_{\text{g}}+\lambda)=
\Bigl(
- q^{ -2  \lambda_2-2  \lambda_3-\lambda_4  +  \frac{c}{2} u_\varphi }
+ \varepsilon ( q^{ -2  \lambda_1-4  \lambda_2-3  \lambda_3-2  \lambda_4 + \frac{5c}{6}u_\varphi } -q^{-\lambda_3-\lambda_4  +\frac{c}{3}u_\varphi })
    \Bigr)   t_\varphi^2
+
\\
\Bigl(
q^{-\lambda_4  +\frac{c}{6} u_\varphi }+q^{-\lambda_3  + \frac{c}{6}u_\varphi }-q^{-2 \lambda_1-2 \lambda_2-2 \lambda_3-\lambda_4  +\frac{c}{2}u_\varphi }-q^{-2 \lambda_1-4 \lambda_2-3 \lambda_3-\lambda_4 + \frac{2c}{3}u_\varphi }+
\\
 \varepsilon (q^{2 \lambda_1+2 \lambda_2+\lambda_3  -\frac{c}{6}u_\varphi } - q^{-\lambda_4  + \frac{c}{6}u_\varphi } - q^{-\lambda_3 + \frac{c}{6}u_\varphi } - q^{-2 \lambda_2-\lambda_3-\lambda_4 +  \frac{c}{3}u_\varphi } + q^{-2 \lambda_1-4 \lambda_2-3 \lambda_3-\lambda_4 +\frac{2c}{3}u_\varphi })
   \Bigr)  t_\varphi
\\
+
\Bigl(
- q^{2  \lambda_1+4  \lambda_2+2  \lambda_3+\lambda_4  -\frac{c}{2}u_\varphi } - q^{2  \lambda_1+2  \lambda_2+\lambda_3+\lambda_4  - \frac{c}{3}u_\varphi } + q^{-2  \lambda_2-\lambda_3  + \frac{c}{6}u_\varphi } - q^{-2  \lambda_1-4  \lambda_2-2  \lambda_3-\lambda_4 + \frac{c}{2}u_\varphi }
\\  \varepsilon (q^{2  \lambda_2+\lambda_3  -\frac{c}{6}u_\varphi } + q^{ 2  \lambda_1+2  \lambda_2+\lambda_3+\lambda_4  - \frac{c}{3}u_\varphi } - q^{-2  \lambda_1-2  \lambda_2-\lambda_3-\lambda_4 + \frac{c}{3}u_\varphi } - q^{-2  \lambda_2-\lambda_3 + \frac{c}{6}u_\varphi })
   \Bigr)
\\
+
\Bigl(
 q^{-2  \lambda_1-2  \lambda_2-\lambda_3  + \frac{c}{6}u_\varphi }-q^{2  \lambda_1+2  \lambda_2+2  \lambda_3+\lambda_4  -\frac{c}{2}u_\varphi } - q^{2  \lambda_2+\lambda_3+\lambda_4   - \frac{c}{3}u_\varphi } +
\\
 \varepsilon (  q^{\lambda_3-\frac{c}{6}u_\varphi } + q^{\lambda_4 - \frac{c}{6}u_\varphi } + q^{2  \lambda_2+\lambda_3+\lambda_4 -\frac{c}{3}u_\varphi }  - q^{-2  \lambda_1-2  \lambda_2-\lambda_3 + \frac{c}{6}u_\varphi } -  q^{2  \lambda_1+4  \lambda_2+3  \lambda_3+\lambda_4 -\frac{2c}{3}u_\varphi }  )
   \Bigr)  t_\varphi^{-1}
\\
+
\Bigl(
q^{2 \lambda_1+4  \lambda_2+3  \lambda_3+2  \lambda_4  -  \frac{5c}{6}u_\varphi } - q^{2  \lambda_2+2  \lambda_3+\lambda_4 - \frac{c}{2}u_\varphi } - q^{ \lambda_3 + \lambda_4 - \frac{c}{3}u_\varphi }
\\
+\varepsilon ( q^{ \lambda_3 + \lambda_4 - \frac{c}{3}u_\varphi } - q^{ 2  \lambda_1+4  \lambda_2+3  \lambda_3+2  \lambda_4 -\frac{5c}{6}u_\varphi  } )
   \Bigr)  t_\varphi^{-2}
 \end{multline*}
 if $\hat{R}=R$. Comparison of the coefficients of $t_\varphi$,  $t_\varphi^2$, $\varepsilon t_\varphi$  and $\varepsilon t_\varphi^2$ on both sides of Eq. \eqref{Ev-id} now leads (upon differentiation at $q=1$) to
the following linearly independent vectors
$v\in Q^\vee$ that are orthogonal to $\lambda-\mu$ if the equality holds:  $(1, 4, 3, 2)$, $(1, 2, 2, 1)$, $(2, 4, 3, 2)$ and $(1, 2, 1, 0)$ if $\hat{R}=R^\vee$, and
$(4,6,4,  1)$,  $(0,2,2,1)$, $(0,0,0,1)$ and $(2,4,2,1)$ if $\hat{R}=R$ (where---recall---the components are with respect to the basis of simple coroots of $R$).

\subsection{Type $\boldsymbol{E}$}
For $R$ of type $E_6$, one has that $\tilde{h}=h=12$, so
$t_\vartheta=\varepsilon q^{-c/12}$  with $\varepsilon=e^{2\pi i/12}$, and the relevant cyclotomic polynomial
is $\Phi_{12}(\varepsilon )=\varepsilon^4-\varepsilon^2+1$. We consider $\hat{m}_\omega (\rho_{\text{g}}+\lambda)$ with $\omega$ being equal either to the minuscule weight $\omega_6$ or to the quasi-minuscule weight $\omega_2=\varphi$. In the minuscule case the LHS of Eq. \eqref{Ev-id} becomes explicitly:
\begin{eqnarray*}
\lefteqn{  \hat{m}_{\omega_6}(\rho_{\text{g}}+\lambda)
=} && \\
&& \varepsilon^{11} \bigl(q^{\frac{1}{3} (-\lambda_1 -2 \lambda_3-\lambda_5 +\lambda_6) +\frac{c}{12}}
   + q^{\frac{1}{3} (-\lambda_1 +\lambda_3-\lambda_5 -2 \lambda_6)   + \frac{c}{12}}\bigr)
\\
&&+ \varepsilon^{10}
\bigl(q^{\frac{1}{3} (-\lambda_1 -2 \lambda_3-3 \lambda_4-\lambda_5 +\lambda_6 )+\frac{c}{6}}
  +  q^{\frac{1}{3} (-\lambda_1 -2 \lambda_3-\lambda_5 -2 \lambda_6 )+\frac{c}{6}} \bigr)
\\
&&+\varepsilon^9\bigl(q^{\frac{1}{3} (-\lambda_1-3 \lambda_2 -2 \lambda_3 -3 \lambda_4-\lambda_5 +\lambda_6) +\frac{c}{4}}
  + q^{\frac{1}{3} (-\lambda_1 -2 \lambda_3-3 \lambda_4-\lambda_5 -2 \lambda_6)  +\frac{c}{4} } \bigr)
\\
&&+\varepsilon^8
\bigl(q^{ \frac{1}{3} (2 \lambda_1+3 \lambda_2 +4 \lambda_3+6 \lambda_4+5 \lambda_5 +4 \lambda_6)  -\frac{2c}{3}}+ \\
&& \qquad   q^{\frac{1}{3} (-\lambda_1-3 \lambda_2 -2 \lambda_3 -3 \lambda_4-\lambda_5 -2 \lambda_6)+\frac{c}{3} }
     +q^{\frac{1}{3} (-\lambda_1-2 \lambda_3 -3 \lambda_4 - 4 \lambda_5 -2 \lambda_6)+\frac{c}{3} }\bigr)
\\
&& +\varepsilon^7
(q^{\frac{1}{3} (2 \lambda_1+3 \lambda_2 +4 \lambda_3+6 \lambda_4+5 \lambda_5  +\lambda_6) -\frac{7c}{12}}
  +q^{\frac{1}{3} (-\lambda_1-3 \lambda_2 -2 \lambda_3 -3 \lambda_4-4 \lambda_5 -2 \lambda_6 )+\frac{5c}{12}})
\\
&&+\varepsilon^6
(q^{\frac{1}{3} (2 \lambda_1+3 \lambda_2 +4 \lambda_3+6 \lambda_4 +2 \lambda_5 +\lambda_6) -\frac{c}{2}}
   + q^{\frac{1}{3} (-\lambda_1-3 \lambda_2 -2 \lambda_3 -6 \lambda_4-4 \lambda_5 -2 \lambda_6) +\frac{c}{2} })
\\
&&+\varepsilon^5
(q^{\frac{1}{3} (-\lambda_1-3 \lambda_2 -5 \lambda_3 -6 \lambda_4-4 \lambda_5 -2 \lambda_6) +\frac{7c}{12}}
  +q^{\frac{1}{3} (2\lambda_1 + 3 \lambda_2 +4 \lambda_3 +3 \lambda_4+2 \lambda_5 +\lambda_6) - \frac{5c}{12} })
\\
&&+\varepsilon^4
(q^{\frac{1}{3} (-4 \lambda_1 -3 \lambda_2 -5 \lambda_3-6 \lambda_4-4 \lambda_5 -2 \lambda_6) +\frac{2c}{3}}
  +\\
 &&\qquad q^{\frac{1}{3} (2\lambda_1 + 4 \lambda_3 +3 \lambda_4+2 \lambda_5 +\lambda_6) -\frac{c}{3}}
   +q^{\frac{1}{3} (2 \lambda_1+3 \lambda_2 +\lambda_3 +3 \lambda_4+2 \lambda_5 +\lambda_6 )-\frac{c}{3} })
\\
&&+\varepsilon^3
(q^{\frac{1}{3} (2 \lambda_1+\lambda_3 +3 \lambda_4+2 \lambda_5 +\lambda_6) -\frac{c}{4} }
   + q^{\frac{1}{3} (-\lambda_1+3 \lambda_2 +\lambda_3 +3 \lambda_4+2 \lambda_5 +\lambda_6 )-\frac{c}{4}})
\\
&&+\varepsilon^2
(q^{\frac{1}{3} (-\lambda_1+\lambda_3 +3 \lambda_4+2 \lambda_5 +\lambda_6) -\frac{c}{6}}
  + q^{\frac{1}{3} (2 \lambda_1+\lambda_3 +2 \lambda_5 +\lambda_6) -\frac{c}{6} })
\\
&&+\varepsilon
(q^{\frac{1}{3} (2 \lambda_1 +\lambda_3-\lambda_5 +\lambda_6) -\frac{c}{12}}
  +q^{\frac{1}{3} (-\lambda_1+\lambda_3 +2 \lambda_5 +\lambda_6) -\frac{c}{12}})
\\
&&+
q^{\frac{1}{3} (-\lambda_1 +\lambda_3-\lambda_5 +\lambda_6) }
+
q^{\frac{1}{3} (-\lambda_1 -2 \lambda_3+2 \lambda_5 +\lambda_6) }
+
q^{\frac{1}{3} (2 \lambda_1+\lambda_3-\lambda_5 -2 \lambda_6 ) },
\end{eqnarray*}
with $\varepsilon^4=\varepsilon^2-1$, $\varepsilon^5=\varepsilon^3-\varepsilon$, $\varepsilon^6=-1$, $\varepsilon^7=-\varepsilon$, $\varepsilon^8=-\varepsilon^2$, $\varepsilon^9=-\varepsilon^3$, $\varepsilon^{10}=1-\varepsilon^2$, and $\varepsilon^{11}=\varepsilon-\varepsilon^3$.
Differentiation at $q=1$ of the coefficients of $\varepsilon^0$, $\varepsilon^1$, $\varepsilon^2$ and $\varepsilon^3$
on both sides of Eq. \eqref{Ev-id} produces the following four linearly independent vectors $v\in Q^\vee$:
$(1, 0, 2, 1, 0, 0)$, $(1, 0, 0, 0, 0, -1)$,  $(1, 0, 2, 2, 2, 1)$ and $(2, 2, 2, 3, 2, 1)$, respectively.
A similar computation for $\omega=\omega_2=\varphi$ complements these with two more linearly independent vectors $v$: $(0, 1, 1, 1, 1, 0)$ and $(0, 1, 1, 3, 1, 0)$, stemming from the coefficients of $\varepsilon^0$ and $\varepsilon^3$.

For $R$ of type $E_7$, one has that $\tilde{h}=h=18$, so
$t_\vartheta=\varepsilon q^{-c/18}$  with $\varepsilon=e^{2\pi i/18}$, and the corresponding cyclotomic polynomial
is $\Phi_{18}(\varepsilon )=\varepsilon^6-\varepsilon^3+1$. We consider $\hat{m}_\omega (\rho_{\text{g}}+\lambda)$ with $\omega$ being equal either to the minuscule weight $\omega_7$ or to the quasi-minuscule weight $\omega_1=\varphi$. In the minuscule case we divide out  an overall factor
$\varepsilon^{1/2}q^{-c/(2h)}$ from Eq. \eqref{Ev-id} before proceeding. The relevant linearly independent vectors $v\in Q^\vee$ are:
$(2, 2, 3, 4, 3, 2, 2)$ ($\varepsilon^0$-term),
$(1, 0, 0, 0, 0, -1, 0)$ ($\varepsilon^1$-term) and
$(0, 1, 0, 2, 3, 2, 1)$ ($\varepsilon^5$-term)  for $\omega=\omega_7$, and
$(1, 1, 2, 2, 2, 1, 0)$  ($\varepsilon^0$-term),  $(1, 0, 1, 2, 1, 1, 0)$ ($\varepsilon^1$-term)
and $(1, 2, 2, 4, 2, 1, 0)$ ($\varepsilon^4$-term)
for $\omega=\omega_1$.

For $R$ of type $E_8$, one has that $\tilde{h}=h=30$, so
$t_\vartheta=\varepsilon q^{-c/30}$  with $\varepsilon=e^{2\pi i/30}$, and the corresponding cyclotomic polynomial
is $\Phi_{30}(\varepsilon )=\varepsilon^8+\varepsilon^7-\varepsilon^5-\varepsilon^4-\varepsilon^3+\varepsilon+1$. We consider $\hat{m}_\omega (\rho_{\text{g}}+\lambda)$ with $\omega$ being equal either  to the quasi-minuscule weight $\omega_8=\varphi$ or to the only other small weight $\omega_1$. The relevant linearly independent vectors $v\in Q^\vee$ are for $\omega=\omega_8$:
$(1, 1, 4, 5, 4, 2, 1, 1)$  ($\varepsilon^0$-term),
$(2, 3, 6, 7, 5, 4, 2, 1)$  ($\varepsilon^1$-term),
$(2, 3, 2, 4, 3, 2, 2, 0)$  ($\varepsilon^2$-term) and
$(0, 0, 2, 2, 1, 0, 1, 0)$  ($\varepsilon^3$-term), and for $\omega=\omega_1$:
$(7, 7, 30, 39, 29, 14, 6, 7)$   ($\varepsilon^0$-term),
 $(14, 21, 44, 51, 35, 28, 12, 5)$  ($\varepsilon^1$-term),
$(16, 24, 16, 30, 23, 14, 15, 0)$   ($\varepsilon^2$-term) and
$(-2, 0, 14, 15, 6, 2, 6, -1)$  ($\varepsilon^3$-term).

\subsection{Type $\boldsymbol{E_7}$ revisited}
In the case that $R$ is of type $E_7$, it follows from the previous considerations that the equality in Eq. \eqref{Ev-id} can hold only if $\lambda-\mu$ is an integral multiple of the weight
\begin{equation}\label{ortho-complement}
\nu=2\omega_1+ 2\omega_2-\omega_3-\omega_4-\omega_5+2\omega_6 -\omega_7=\alpha_1+\alpha_2+\alpha_6
\end{equation}
(which spans the orthogonal complement of the hyperplane spanned by the above vectors $v\in Q^\vee$ for this case).
Substituting $\mu=\lambda +k \nu$ ($k\in\mathbb{Z}$) and application of the operator $(q\frac{ \text{d}}{\text{d}q})^2$ to the coefficients on both sides  of the equality entails a system of quadratic relations in $\lambda$ and $k$ (upon evaluation at $q=1$).
In each of these relations the LHS cancels against the quadratic terms in $\lambda$ on the RHS (viz. the $k^0$-terms) and ---more surprisingly---the quadratic terms in $k$ on the RHS also turn out to cancel against each other. From the remaining linear terms in $k$ we then deduce that the
equality in Eq. \eqref{Ev-id} implies that either $k=0$ or that $\lambda$ must satisfy a nonhomogenous system of five
linearly independent equations:
$2\lambda_1+ 2\lambda_2+ 3\lambda_3+4\lambda_4+ 3\lambda_5+ 2\lambda_6+ 2\lambda_7=c$ ($\varepsilon^0$-term),
$\lambda_1-\lambda_6=0$ ($\varepsilon^1$-term) for $\omega=\omega_7$, and
 $\lambda_1+\lambda_2+2\lambda_3+ 2\lambda_4+ 2\lambda_5+\lambda_6= c/2$ ($\varepsilon^0$-term),
  $\lambda_1+\lambda_3+ 2\lambda_4+ \lambda_5+ \lambda_6=c/3$ ($\varepsilon^1$-term) and
  $2\lambda_2+\lambda_3+ 2\lambda_4+ \lambda_5 =c/3$ ($\varepsilon^5$-term) for $\omega=\omega_1$.
The intersection of its two-dimensional plane of solutions with the convex hull of $P_c$ is given by the triangle
\begin{equation}\label{triangle}
\lambda^{(0)} -s\nu -r\eta,\quad  |r|\leq s\leq \frac{c}{12},
\end{equation}
where
$\lambda^{(0)} :=\frac{c}{6}(\omega_1+\omega_2+\omega_6)$, $\nu$ is given by Eq. \eqref{ortho-complement}, and $\eta:=\omega_3-\omega_5$.
Our condition that $c$ not be an integral multiple of $6$ when $R$ is of type $E_7$ guarantees that the intersection of the triangle with $P_c$ is empty, i.e. the equality in Eq. \eqref{Ev-id} can only hold if $k=0$ (so $\lambda=\mu$).

\begin{remark}\label{E7-degenerations} When $c$ is a multiple of $6$ the intersection of the triangle \eqref{triangle} with $P_c$ is given by weights of the form
$\lambda^{(0)} -k\nu -l\eta$ with $k,l\in\mathbb{Z}$ such that $ |l|\leq k\leq [\frac{c}{12}]$.
For instance, for $c=6$ the intersection consists only of $\lambda^{(0)}$ (so degenerations
are not possible in this case) whereas for proper multiples of $6$ a pair of weights $\lambda$ and $\mu$ in the triangle
corresponding to the same value for $l$ and different values for $k$ may lead to equal expressions on both sides of Eq. \eqref{Ev-id} for all $\omega\in\hat{P}$ small. Explicit computations for a few multiples of $6$ suggest that for fixed $l$ and any $\omega\in\hat{P}$ small, the expression for $\hat{m}_\omega (\rho_{\text{g}}+\lambda^{(0)} -k\nu -l\eta )$ is in fact independent of $k=|l|,\ldots ,[\frac{c}{12}]$.
Such degenerations only occur for weights near the affine wall of $P_c$. Indeed, since $\langle \lambda^{(0)} -s\nu -r\eta,\varphi^\vee\rangle\geq
\langle\lambda^{(0)}-\frac{c}{12}\nu,\varphi^\vee\rangle= \frac{11}{12}c$ for $|r|\leq s\leq \frac{c}{12}$, the degenerations in question are restricted to weights outside
$P_{\tilde{c}}\subseteq P_c$ with $\tilde{c}=\lceil\frac{11}{12}c\rceil$.
\end{remark}

\section*{Acknowledgments.}
The computations sustaining our case-by-case analysis to verify the nondegeneracy of the spectrum of the Macdonald operators with unitary parameters for the exceptional root systems benefitted much from
Stembridge's Maple packages {\tt COXETER} and
{\tt WEYL}.

\bibliographystyle{amsplain}

\end{document}